\documentclass[a4paper, 12pt, one column, aas_macros]{amsart}
\usepackage[english]{babel}
\usepackage[english]{babel}
\usepackage{graphicx}
\usepackage[latin1]{inputenc}
\usepackage{amsmath,amssymb,hyperref,srcltx,mathtools}
\usepackage{amsfonts}%
\usepackage[dvips]{geometry}
\usepackage[normalem]{ulem}
\usepackage{tikz-cd}
\usepackage[shortlabels]{enumitem} 

\newcommand{\Oc}{\mbox{${\mathcal O}$}}
\newcommand{\Oh}{\mbox{$\hat{\mathcal O}$}}
\newcommand{\diff}[2]{\mbox{{\rm Diff}{${\,}_{#1}({\mathbb C}^{#2},0)$}}}
\newcommand{\diffh}[2]{\mbox{$\widehat{\rm Diff}{{}_{#1}({\mathbb C}^{#2},0)}$}}

\newcommand{\cn}[1]{\mbox{(${\mathbb C}^{#1},0$)}}

\newtheorem{pro}{Proposition}

\newtheorem{cor}{Corollary}
\newtheorem{rem}{Remark}

\DeclareMathOperator{\fix}{Fix}

\usepackage{xwatermark}
\usepackage{xcolor}

\newtheorem{lemma}{Lemma}

\newtheorem{thm}{Theorem}
\newtheorem*{theorem*}{Theorem}

\newtheorem*{main theorem*}{Fixed Point Curve Theorem}
\newtheorem*{P-D*}{Poincar\'{e}-Dulac normal form}
\newtheorem*{hsmtb*}{Holomorphic Stable Manifold Theorem for Diffeomorphisms}
\newtheorem*{sierpinski*}{Sierpi\'nski Theorem}
\newtheorem*{hsmtv*}{Holomorphic Stable Manifold Theorem for Vector Fields}
\newtheorem{example}{Example}

\theoremstyle{remark}
\newtheorem{defi}{Definition}

\newcommand{\C}{\mathbb{C}}

\newcommand{\nn}{\mathbb{N}}  
\newcommand{\zz}{\mathbb{Z}} 
\newcommand{\cc}{\mathbb{C}}

\newcommand{\codim}{\mathrm{codim}}
\newcommand{\mm}{\mathfrak{m}}
\newcommand{\fg}{{\mathfrak T}_{G} }

\newcommand{\spec}{\mathrm{spec}}

\title[Stability of closed leaves holomorphic foliations]{Stability of closed leaves holomorphic foliations via   torsion behavior of groups of germs }

\author{Javier Rib\'{o}n}
\address{Instituto de Matem\'{a}tica e Estat\'\i stica \\
Universidade Federal Fluminense\\
Campus do Gragoat\'a\\
Rua Marcos Valdemar de Freitas Reis s/n, 24210\,-\,201 Niter\'{o}i, Rio de Janeiro - Brasil }

\email{jribon@id.uff.br}
 
\thanks{MSC-class. Primary: 32M25, 34M45, 37F75; Secondary: 34C45, 37C85, 37C25, 32H50} 
\begin{document}

\begin{abstract}
Consider a compact leaf $\mathcal{L}$ of a holomorphic foliation $\mathcal{F}$ such that
all the leaves of the restriction of $\mathcal{F}$ to some neighborhood $U$ of $\mathcal{L}$ 
are closed in $U$. We show that there exists an invariant germ of analytic set $V$ 
in a neighborhood of $\mathcal{L}$, of dimension higher than $\dim (\mathcal{F})$, 
consisting of compact leaves, such that the volume of 
the leaves in $V$ is uniformly bounded by above. 
We also provide a globalization of this result if the ambient manifold is projective. 

In order to study closed leaves foliations, and via the holonomy representation,  
we introduce a new concept, of independent interest, 
for  subgroups $G$ of germs of holomorphic diffeomorphisms, the so called {\it torsion locus}.
We show that it is non-trivial if $G$ has finite orbits. 
 
\end{abstract}
\maketitle
\tableofcontents

\section{Introduction}\label{sec:introduction}
The main goal of this paper is introducing and proving a semi-local result on the {\it periodic orbit conjecture} in the holomorphic category.
The conjecture, due to Reeb, claims that if all leaves of a regular foliation in a compact manifold are compact sets then the volume of leaves is uniformly bounded, or
equivalently, each leaf $\mathcal{L}$ is stable, i.e. there exists a neighborhood basis of $\mathcal{L}$ whose sets are invariant by the foliation
\cite{Epstein:leaves_compact}.
This is also equivalent to the finiteness of each holonomy group associated to any leaf \cite{Epstein:leaves_compact}.
The conjecture does not hold in general, the first
counterexample to the conjecture is due to Sullivan \cite{Sullivan:periodic}. 
Even in the real analytic setting there are counterexamples by 
Thurston (cf. \cite{Sullivan:periodic}) and Epstein and Vogt \cite{Epstein-Vogt:counterexample_dim_3}.

It is not known whether the conjecture holds in the general holomorphic setting, even
if it does for K\"{a}hler manifolds \cite{Edwards-Millet-Sullivan} by a result of Edwards, Millet and Sullivan.
Moreover, Pereira showed a local to global stability result in the K\"{a}hler case \cite{Pereira:stability}.
In contrast, Holmann provided a negative result, proving that a 
holomorphic counterexample to the conjecture cannot be provided by  the action of a holomorphic vector field 
\cite[Holmann]{Holmann:anal_period}  \cite{Holmann:stab_hol}.

\strut

In this paper, we approach the holomorphic problem by considering foliations whose leaves are closed 
in a neighborhood of a given compact leaf  $\mathcal{L}$.
\begin{defi}
    \label{def:closed_leaves}
    We say that $(M, \mathcal{F}, {\mathcal L})$ (or just $(\mathcal{F}, {\mathcal L})$ if $M$ is implicit)
    is a {\it closed leaves foliation} (at $\mathcal{L}$)  if $\mathcal{L}$ is a compact leaf of a holomorphic foliation $\mathcal{F}$
    defined in a complex manifold $M$ and there exists an open neighborhood  $U$ of $\mathcal{L}$ that satisfies that all  all leaves of the 
    restriction $\mathcal{F}_{U}$ of $\mathcal{F}$ to $U$ are closed in $U$. 
\end{defi}
The study of closed leaves foliations is useful for the {\it periodic orbit conjecture} because if $M$ and all leaves of $\mathcal{F}$ are compact then 
the leaves of $\mathcal{F}_{U}$ are closed for any  open set $U$. 
The leaves of $\mathcal{F}_U$ 
are obtained by considering the connected components of $\mathcal{L}' \cap U$ for any leaf $\mathcal{L}'$ of $\mathcal{F}$.
Let us remark that the closed leaves hypothesis is weaker than the compact leaves hypothesis. 
Indeed, $\mathcal{F}$ could be a singular foliation with singularities outside of 
$\overline{\mathcal L} = \mathcal{L}$. Moreover, in non-compact manifolds, 
the leaves of a closed leaves foliation are not necessarily compact sets.  

\strut

Proving the conjecture amounts to show that the ambient manifold is stable, i.e. consisting of stable leaves (cf. Definition \ref{def:stable}). 
The main result of the paper is a step forward in this direction: there exists a 
stable analytic set in intermediate dimension,  
i.e. higher than $\dim (\mathcal{F}) = \dim (\mathcal{L})$ but maybe smaller than 
$\dim (M)$. 

\begin{thm}
\label{thm:main}
Let $(\mathcal{F}, {\mathcal L})$ be a closed leaves foliation at $\mathcal{L}$.
Then there exists an $\mathcal{F}$-invariant germ of analytic set $(V, {\mathcal L})$ 
of dimension greater than $\dim (\mathcal{F})$, 
such that any leaf of $\mathcal{F}_{V}$ is stable.
In particular all  leaves of $V$ are compact.
Moreover, $V$  
is the saturated of $\mathfrak{T}_{\mathcal{H}_{\mathcal{F},\mathcal{L}}}$
with respect to $\mathcal{F}$.
\end{thm}
We will explain the last sentence of the result soon. 
We will find $V$ such that the holonomy group of $\mathcal{F}_{V}$ associated to $\mathcal{L}$ is finite. 
Stability follows from the Reeb local stability theorem (cf. \cite{Wu-Reeb,Camacho-Lins_Neto:foliations}).
%

We can globalize the invariant set $V$ in Theorem \ref{thm:main} if $M$ is projective. 
\begin{thm}
    \label{thm:main_p}
    Let $(M, \mathcal{F}, {\mathcal L})$ be a closed leaves foliation, where $M$ is a complex projective manifold. 
    Then there are $\mathcal{F}$-invariant irreducible projective varieties $W_1, \hdots, W_k$, containing $\mathcal{L}$,
    and dominant rational first integrals 
    $f_j: W_j \to S_j$, where 
    $S_j$ is a projective submanifold of positive dimension of an irreducible component of the Hilbert scheme of varieties
    and the general $f_j$-fiber is a leaf of $\mathcal{F}$ for any $1 \leq j \leq k$.
    Moreover, the germ $(W_1 \cup \hdots \cup W_k , {\mathcal L})$ is the saturated of $\mathfrak{T}_{\mathcal{H}_{\mathcal{F},\mathcal{L}}}$
    with respect to $\mathcal{F}$.
\end{thm}
In order to prove Theorem \ref{thm:main}, we consider the holonomy representation \cite{Ehresmann-Shih:stabilite}  (cf. \cite[section IV.1]{Camacho-Lins_Neto:foliations})
\[ \Phi_{\mathcal{F}, \mathcal{L}}: \pi_{1} (\mathcal{L}, {\bf z}_{0}) \to \mathrm{Diff}(T,{\bf z}_{0}), \]
and the holonomy subgroup $\mathcal{H}_{\mathcal{F},\mathcal{L}} := \Phi_{\mathcal{F}, \mathcal{L}}( \pi_{1} (\mathcal{L}, {\bf z}_{0}))$ of
the group $\mathrm{Diff}(T,{\bf z}_{0})$, of 
germs of holomorphic diffeomorphisms defined  in a neighborhood of ${\bf z}_{0}$ in a transverse section $T$, where $T \cap \mathcal{L} = \{{\bf z}_{0}\}$.
We can consider $\mathcal{H}_{\mathcal{F},\mathcal{L}}$ as a subgroup of $\diff{}{n}$
by identifying the germ $(T,{\bf z}_{0})$ with $(\mathbb{C}^{n}, {\bf 0})$, where $n = \codim (\mathcal{F})$.

Given a 
closed leaves foliation $(\mathcal{F}, \mathcal{L})$, the intersection of a leaf with a compact subset of a transverse section $T$
is always a finite set (Proposition \ref{pro:closed_finite}). So if we consider the pseudogroup (cf. Definition \ref{def:pseudogroup})
of holonomy (cf. Definition \ref{def:hol_pseudo}) $\mathcal{G}$ of $\mathcal{F}_{U}$ in $T$, where $U$ is a small neighborhood of $\mathcal{L}$, it has finite orbits. 

The analogue of a closed leaves foliation in the context of $\diff{}{n}$ is a finite orbits subgroup. 
Let $G < \diff{}{n}$, we say that $G$ has the {\it finite orbits property} if it is induced by a finite orbits  pseudogroup 
$\mathcal{G}$ in an open subset $U$ of $\C^{n}$, i.e.  $G$ is the group of germs at ${\bf 0}$ of maps 
$\phi: V \to W$ in $\mathcal{G}$, where $V, W \subset U$, such that $\phi ({\bf 0}) = {\bf 0}$
(Definition \ref{def:fin_orb_prop}). The above discussion implies that 
the holonomy group $\mathcal{H}_{\mathcal{F},\mathcal{L}}$ has finite orbits if 
 $(\mathcal{F}, \mathcal{L})$ is a closed leaves foliation.

Now, already in the setting of groups of germs of biholomorphisms, we introduce in this work   
a new concept for subgroups $G$ of $\diff{}{n}$ that would be central to prove 
Theorems \ref{thm:main} and \ref{thm:main_p}, the so called {\it torsion locus} $\fg$ of $G$. 
It is the maximal germ   $\mathfrak{T}$ of analytic set at ${\bf 0}$ such that 
$G_{|\mathfrak{T}}$ consists of finite order elements (Definition \ref{def:fg} and Proposition \ref{pro:fg}). 
\begin{example}
We can associate to a given $\phi \in \diff{}{n}$ the so called {\it index of embeddability} $\mm (\phi)$ of $\phi$ (Definition \ref{def:embed}).
One of its properties is that if $\phi_{|V}$ has finite order, where $V$ is a germ of analytic set at ${\bf 0}$, then 
$V \subset \fix (\phi^{\mm (\phi)})$ (Corollary \ref{cor:per_to_fix}). 
Thus, given a subgroup $G$ of $\diff{}{n}$, we have $\fg = \cap_{\phi \in G} \fix (\phi^{\mm (\phi)})$.
In the particular case in which $\phi$ is tangent to the identity, i.e. $D_{\bf 0} \phi = \mathrm{id}$, the index $\mm (\phi)$ is equal to $1$. 
So, we obtain $\fg = \fix (G) = \cap_{\phi \in G} \fix (\phi)$ if $G$ consists of tangent to the identity elements.

The torsion locus is  trivial for $\mathfrak{T}_{\langle z+ z^{2} \rangle} = \{ {\bf 0} \}$, where $z + z^{2} \in \diff{}{}$.
It coincides with $(\C^{n}, {\bf 0})$ if and only if all elements of $G$ have finite order. 
\end{example}
%
 
The key property of finite orbits subgroups of  $\diff{}{n}$ is that their torsion loci are non-trivial.
\begin{thm}
\label{thm:main_g}
Let $G$ be a finite orbits  subgroup of $\diff{}{n}$. Then $G$ is virtually solvable, 
its torsion locus $\fg$  
has positive dimension and is  
$G$-invariant, and  $G_{|\fg}$  is virtually abelian and consists of 
finite order elements. Moreover, $\fg$ is maximal among the germs of analytic variety $V'$ (and also among the formal varieties) 
such that  $V'$ is $G$-invariant and  $G_{|V'}$  consists of 
finite order elements. 
%
\end{thm}
Since $\mathcal{H}_{\mathcal{F},\mathcal{L}}$ is finitely generated, because $\pi_{1} (\mathcal{L}, {\bf z}_{0})$ is finitely generated, we can focus in 
finitely generated subgroups of $\diff{}{n}$.
\begin{thm}
\label{thm:main_g_fg}
Let $G$ be a finitely generated subgroup of $\diff{}{n}$ that has the finite orbits property. Then $\fg = \fix (\overline{G}_{0} \cap G)$ holds,
$\dim (\fg) \geq 1$,  
$G_{|\fg}$  is finite and a quotient of the finite group $G / (\overline{G}_{0} \cap G)$ (or equivalently $\overline{G} / \overline{G}_{0}$
or $\overline{D_{\bf 0} G} / \overline{D_{\bf 0} G}_{0}$). In particular, we have 
$\sharp G_{|\fg} \leq |G : \overline{G}_{0} \cap G| =  |\overline{G} : \overline{G}_{0}| =   |\overline{D_{\bf 0} G} : \overline{D_{\bf 0} G}_{0}|$.    
\end{thm}
Interestingly, the action of $G$ in $\fg$ is limited by its linearized $D_{\bf 0} G$ at the origin (cf. Definition \ref{def:linearized}). Indeed, 
$G_{|\fg}$ is a quotient of the group $\overline{D_{\bf 0} G} / {\overline{D_{\bf 0} G}}_{0}$ that depends only 
on the linear parts at the origin of the elements of $G$, since $\overline{D_{\bf 0} G}$ is the Zariski-closure of $D_{\bf 0} G$
in $\mathrm{GL} (\C^{n})$ and ${\overline{D_{\bf 0} G}}_{0}$ is its connected component of identity.
We remark that the group $\overline{G}$ is a subgroup of the group of formal diffeomorphisms $\diffh{}{n}$ (cf. Definition \ref{def:lin_phi}), 
the so called Zariski-closure of $G$ and $\overline{G}_0$ is its connected component of identity (cf. Definition \ref{def:Zariski-closure});
these concepts are discussed later on.  

The finite orbits property for $G < \diff{}{n}$ does not imply that $D_{\bf 0} G$ has finite orbits, or, in particular, that 
all eigenvalues of the linear parts of elements of $G$ at ${\bf 0}$ are roots of unity \cite[Theorem 1]{Lisboa-Ribon:fixed-point}. 
Nevertheless, the following immediate corollary 
of Theorem \ref{thm:main_g} shows that it has an influence on the spectrum of elements of $D_{\bf 0} G$.
\begin{cor}
\label{cor:main_eigen}
Let $G$ be a finite orbits  subgroup of $\diff{}{n}$. Then $\spec (D_{\bf 0} \phi)$ contains a root of unity for any $\phi \in G$.
\end{cor}
The result is sharp, since it is possible that the root of unity eigenvalue is unique \cite[Theorem 1]{Lisboa-Ribon:fixed-point} (Remark \ref{rem:just_1_eigen}).

Theorems \ref{thm:main_g} and \ref{thm:main_g_fg} generalize well-known results for the one-dimensional case ($n=1$) 
by Mattei and Moussu \cite{MaMo:Aen}. Indeed, Theorem \ref{thm:main_g}  implies that a finite orbits subgroup of $\diff{}{}$ 
satisfies $\fg = \cn{}$ and hence consists of finite order elements and is abelian (cf. Lemma \ref{lem:fin_orb_1}). 
If, moreover, $G$ is finitely generated then $G$ is a finite abelian group (cf. also Lemma \ref{lem:fin_orb_1_fg}).

Let us mention that Theorem \ref{thm:main_g_fg} for the cyclic case and $n=2$ is, essentially, the main theorem of \cite{Lisboa-Ribon:fixed-point}.
We present an example in section \ref{section:infinite} that shows that 
Theorem \ref{thm:main_g} cannot be strengthened to obtain that $G_{|\fg}$  is finite.  
Next, we consider some cases in which the properties of the linearized $D_{\bf 0} G$ of $G$ at the origin of
a finite orbits group force the torsion locus $\fg$ of $G$ to coincide with $\fix (G)$.

\begin{cor}
    \label{cor:main_g}
     Let $G$ be a subgroup of $\diff{}{n}$ that has the finite orbits property. Then, we obtain that $\fix (G)$ is equal to $\fg$ and of
     positive dimension if
     \begin{itemize}
     \item $G$ is finitely generated and  $\overline{G} =\overline{G}_0$ or
     \item $\overline{\langle \phi \rangle} =\overline{\langle \phi \rangle}_0$ for any $\phi \in G$ or
     \item $D_{\bf 0} \phi$ is unipotent   for any $\phi \in G$.
     \end{itemize}
\end{cor}
Theorem \ref{thm:main} is a direct consequence of Theorem \ref{thm:main_g_fg} since $\mathcal{H}_{\mathcal{F},\mathcal{L}}$ is finitely generated
and has finite orbits if $(\mathcal{F}, {\mathcal L})$ is a closed leaves foliation. 
Indeed, $V$ is the saturated of $\mathfrak{T}_{\mathcal{H}_{\mathcal{F},\mathcal{L}}}$ (cf. Definition \ref{def:saturated})
and has dimension $\dim (\mathcal{F}) + \dim (\mathfrak{T}_{\mathcal{H}_{\mathcal{F},\mathcal{L}}}) > \dim (\mathcal{F})$. 
Theorem \ref{thm:main_p} is derived from Theorem \ref{thm:main}, the globalization procedure borrows ideas by G\'{o}mez-Mont
\cite{Gomez-Mont:integral_compact}.

We introduce Zariski-closed subgroups of formal diffeomorphisms in section \ref{sec:pro-algebraic}.
Their properties will be used throughout the paper. The finite orbits property for subgroups of $\diff{}{n}$ is 
introduced in section \ref{sec:finite_orbits}.
We deal with finite orbits cyclic subgroups of $\diff{}{n}$ in section \ref{sec:cyclic}. There, we 
improve the main result of \cite{Lisboa-Ribon:fixed-point} to prepare the study of general finite 
orbits subgroups of $\diff{}{n}$ in section \ref{sec:general}, where we show Theorems \ref{thm:main_g} and \ref{thm:main_g_fg}.
We profit from the work in diffeomorphisms to show Theorems \ref{thm:main} and \ref{thm:main_p} on closed leaves foliations in 
section \ref{sec:closed_leaves}. 

 \section{Pro-algebraic groups}
\label{sec:pro-algebraic}
 We consider pro-algebraic (also called Zariski-closed) groups of formal diffeomorphisms since they are very useful to study invariance
 properties \cite[section 7.2.7.1]{Ribon:Zariski-closure} and in particular the torsion locus $\fg$.
 \subsection{Preliminaries}
 We introduce here some useful concepts and properties 
 for the sake of completeness.
 \begin{defi}
\label{not:basic}
We denote by $\Oc_{n} = \C \{z_1, \hdots, z_n\}$ (resp. $\Oh_{n} = \C [[z_1, \hdots, z_n]]$) the local rings of convergent (resp. formal) 
complex power series in $n$ variables, centered at ${\bf 0}$. 
We denote by $\mathfrak{m}_{n}$ and  $\hat{\mathfrak{m}}_{n}$ respectively their maximal ideals.
\end{defi}
\begin{defi}
\label{def:diff_n}
We denote by $\diff{}{n}$ the group of germs of complex analytic diffeomorphisms at ${\bf 0} \in \cc^{n}$.
It consists of germs of biholomorphisms $\phi: U \to V$, where $U, V$ are open subsets of $\cc^{n}$ containing ${\bf 0}$ and 
$\phi ({\bf 0}) = {\bf 0}$.
\end{defi}
\begin{rem}
As a consequence of the inverse function theorem, and through the Taylor series expansion at ${\bf 0}$, any $\phi \in \diff{}{n}$ is of the form 
\begin{equation}
    \label{equ:dif_taylor}
    \phi (z_1, \hdots, z_n) = \left( \sum_{{\bf j} \in \zz_{\geq 0}^{n}, \, |{\bf j}| \geq 1} a_{{\bf j}; 1} {\bf z}^{{\bf j}} , \hdots, 
    \sum_{{\bf j} \in \zz_{\geq 0}^{n}, \, |{\bf j}| \geq 1}a_{{\bf j}; n} {\bf z}^{{\bf j}} \right) \in \underbrace{ {\mathfrak m}_{n} \times \hdots \times {\mathfrak m}_{n} }_{n \ \text{times}}  
\end{equation}
and its linear part 
$( \sum_{|{\bf j}| = 1} a_{{\bf j}; 1} {\bf z}^{{\bf j}} , \hdots, 
\sum_{|{\bf j}| = 1}a_{{\bf j}; n} {\bf z}^{{\bf j}} )$ at ${\bf 0}$ is a linear isomorphism, where 
${\bf z} = (z_1, \hdots, z_n)$, ${\bf j} = (j_1, \hdots, j_n)$, $|{\bf j}| = j_1+ \hdots+ j_n$ and 
$ {\bf z}^{{\bf j}} = z_{1}^{j_1} \hdots z_{n}^{j_n}$.
\end{rem}
The previous remark can be used as a motivation to define the group of formal diffeomorphisms.
\begin{defi}
     \label{def:lin_phi}
     We denote by $\diffh{}{n}$ the group of formal diffeomorphisms in dimension $n$.
     It consists of expressions as \eqref{equ:dif_taylor} but where $\phi$ belongs to the cartesian product
     ${(\hat{\mathfrak m}_{n})}^{n}$ 
     and its linear part 
     $D_{\bf 0} \phi := ( \sum_{|{\bf j}| = 1} a_{{\bf j}; 1} {\bf z}^{{\bf j}} , \hdots,  \sum_{|{\bf j}| = 1}a_{{\bf j}; n} {\bf z}^{{\bf j}} )$ at ${\bf 0}$ 
     is a linear isomorphism.
     %
\end{defi}
\begin{rem}
Given $\phi \in \diffh{}{n}$ and $k \in \nn$, we can define its $k$-jet 
\[ (j^{k} \phi) (z_1, \hdots, z_n) = \left( \sum_{{\bf j} \in \zz_{\geq 0}^{n}, \, |{\bf j}| \leq k} a_{{\bf j}; 1} {\bf z}^{{\bf j}} , \hdots, 
\sum_{{\bf j} \in \zz_{\geq 0}^{n}, \, |{\bf j}| \leq k}a_{{\bf j}; n} {\bf z}^{{\bf j}} \right)  ; \]
it belongs to $\diff{}{n}$. Given $\phi, \eta \in \diffh{}{n}$, we can define $\phi \circ \eta \in \diffh{}{n}$ as the formal diffeomorphism
whose $k$-jet coincides with the $k$-jet of $(j^{k} \phi) \circ (j^{k} \eta)$ for any $k \in \nn$. Analogously, we can define
$\phi^{-1} \in \diffh{}{n}$ as the formal diffeomorphism whose $k$-jet coincides with the $k$-jet of ${(j^{k} \phi)}^{-1}$ for any
$k \in \nn$. 
The set $\diffh{}{n}$ is a group with this operation.
\end{rem}
\begin{defi}
    \label{def:linearized}
    Given $G < \diffh{}{n}$, we define its linearized $D_{\bf 0} G$ at ${\bf 0}$  as the linear group 
    $D_{\bf 0} G = \{ D_{\bf 0} \phi : \phi \in G \}$ (cf. Definition \ref{def:lin_phi}).
\end{defi}
 \begin{defi}
     \label{def:gk}
     Let $k \in \nn$. 
     Given $\phi \in \diffh{}{n}$, it defines a map $\phi_{k} \in \mathrm{GL} (\hat{\mathfrak m}_{n}/\hat{\mathfrak m}_{n}^{k+1})$
     such that $\phi_{k} (f + \hat{\mathfrak m}_{n}^{k+1}) = f \circ \phi + \hat{\mathfrak m}_{n}^{k+1}$
     for any $f \in \Oh_{n}$. Note that $\hat{\mathfrak m}_{n}/\hat{\mathfrak m}_{n}^{k+1}$ is a finite dimensional complex vector space for any 
     $k \in \nn$.
     Given a subgroup $G$ of $\diffh{}{n}$, we define
     \[ G_k = \{ \phi_k : \phi \in G \} . \]
     It is a group of automorphisms of the $\cc$-algebra $\hat{\mathfrak m}_{n}/\hat{\mathfrak m}_{n}^{k+1}$ 
    (cf.  \cite[section 7.2.1]{Ribon:Zariski-closure}).
    We define $\overline{G}_{k}$ as the Zariski-closure of $G_k$ in $\mathrm{GL} (\hat{\mathfrak m}_{n}/\hat{\mathfrak m}_{n}^{k+1})$.
 \end{defi}
 \begin{defi}
     \label{def:krull}
     The Krull topology of $\Oh_{n}$ is the topology defined by the neighborhood base ${\{\mathcal{B}_{f,k}\}}_{f \in \Oh_{n}, \, k \in \nn}$, where
     $\mathcal{B}_{f, k} = f +  \hat{\mathfrak m}_{n}^{k+1}$. 
     It induces a topology in the cartesian product $(\hat{\mathfrak m}_{n})^{n}$  and thus in $\diffh{}{n}$.
     Given $G < \diffh{}{n}$, we have that
     \[ \overline{G}^{\mathrm{krull}} = \{ \phi \in \diffh{}{n} : \phi_k \in G_k \, \forall k \in \nn \}.  \]
     is the closure of $G$ in the Krull topology.
 \end{defi}
 \begin{example}
     The Krull closure of $\diff{}{n}$ is equal to $\diffh{}{n}$.
 \end{example}
 \begin{rem}
     \label{rem:dual}
     The map $f \to f^{*}$ that sends each element to its dual is an anti-isomorphism of algebraic matrix groups 
     from $\mathrm{GL}(T_{\bf 0} \C^{n})$ to $\mathrm{GL}(T_{\bf 0}^{*} \C^{n})$.
     The image of $D_{\bf 0} G$ by such anti-isomorphism is the group $G_1$.
 \end{rem}
 \begin{rem}
     Since $G_k$ is a group of automorphisms of the $\cc$-algebra $\hat{\mathfrak m}_{n}/\hat{\mathfrak m}_{n}^{k+1}$, so is 
     $\overline{G}_k$ \cite[Remark 7.2.12]{Ribon:Zariski-closure}, i.e. we have
     \[ A(fg + \hat{\mathfrak m}_{n}^{k+1}) = A(f + \hat{\mathfrak m}_{n}^{k+1}) A (g + \hat{\mathfrak m}_{n}^{k+1}) \]
     for all $f, g \in \Oh_{n}$. In particular, any $A \in \overline{G}_k$ is of the form 
     $\phi_k$ for some $\phi \in \diffh{}{n}$ \cite[Lemma 2.1]{JR:solvable25}.
 \end{rem}
 \begin{defi}
     \label{def:Zariski-closure}
     Let $G$ be a subgroup of $\diffh{}{n}$. We define its Zariski-closure $\overline{G}$ as the inverse limit $\varprojlim \overline{G}_{k}$. 
     Equivalently, we have
     \[ \overline{G} = \{ \phi \in \diffh{}{n} : \phi_k \in \overline{G}_{k} \ \forall k \in \nn \} . \]
     We define its connected component of identity $\overline{G}_{0}$ as the group 
     \[ \overline{G}_{0} = \{ \phi \in \diffh{}{n} : \phi_k \in \overline{G}_{k,0} \ \forall k \in \nn \} , \]
     where $\overline{G}_{k,0}$ is the connected component of identity of $\overline{G}_{k}$.
 \end{defi}
 \begin{defi}
     We say that a subgroup $G$ of $\diffh{}{n}$ is pro-algebraic (or Zariski-closed) if $\overline{G} = G$.  
 \end{defi}
 \begin{rem}
     \label{rem:zar_imp_krull}
     Any pro-algebraic subgroup of $\diffh{}{n}$ is Krull-closed.
 \end{rem}
 \begin{rem} 
     Let $G < \diffh{}{n}$.
     Then $\overline{G}$ is the minimal 
     pro-algebraic group containing $G$ \cite[Remark 7.2.16]{Ribon:Zariski-closure}. 
 \end{rem}
 The connected component of identity of a pro-algebraic group has analogous properties to those of the connected component of identity
 of an algebraic matrix group.
 \begin{rem}
     \label{rem:g0}
     Let $G < \diffh{}{n}$.
     Then $\overline{G}_0$ is a finite index normal pro-algebraic subgroup of $\overline{G}$ that satisfies
     $\overline{G}_{0} = \{ \phi \in \overline{G} : \phi_1 \in \overline{G}_{1,0} \}$.
     \cite[Proposition 2.3 and Remark 2.9]{JR:solvable25}. 
     In particular, the projection map 
     \[
     \begin{array}{ccc}
      \overline{G}/\overline{G}_0 & \to & \overline{G}_{1}/ \overline{G}_{1,0} \\
      \phi + \overline{G}_0 & \mapsto & \phi_1 + \overline{G}_{1,0}
     \end{array}
     \]
     is an anti-isomorphism of groups and hence $|\overline{G} : \overline{G}_0| = |\overline{G}_{1}  : \overline{G}_{1,0}|$.
     Thus, the linear part map  
      \[
      \begin{array}{ccc}
      \overline{G}/\overline{G}_0 & \to & \overline{D_{\bf 0} G} / {\overline{D_{\bf 0} G}}_{0} \\
      \phi + \overline{G}_0 & \mapsto & D_{\bf 0} \phi  +  {\overline{D_{\bf 0} G}}_{0}
     \end{array}
     \]
     is an isomorphism of groups and  $|\overline{G} : \overline{G}_0| = |\overline{D_{\bf 0} G} : {\overline{D_{\bf 0} G}}_{0}|$ by
     Remark \ref{rem:dual}.
\end{rem}
\begin{rem}
     \label{rem:fin_ind}
     Let $G < \diffh{}{n}$. Any finite index subgroup of $\overline{G}$ contains $\overline{G}_0$ \cite[Lemma 2.1]{JR:finite}. 
     Moreover, if $H$ is a finite index normal subgroup of $G$ then $\overline{H}$ is a finite index normal subgroup of $\overline{G}$, 
     the natural map $G/H \to \overline{G}/\overline{H}$ is surjective and 
     $|\overline{G}: \overline{H}| \leq |G:H|$ \cite[Lemma 2.4]{JR:finite}. In particular $\overline{H}$ contains $\overline{G}_0$. 
 \end{rem}
 \begin{rem}
     \label{rem:conn_char}
     Since the natural map $G \to \overline{G}/\overline{G}_0$ is surjective, the homomorphism of groups
     $G/(G \cap \overline{G}_0) \to \overline{G}/\overline{G}_0$ is an isomorphism  and 
     $|G : \overline{G}_{0} \cap G| =  |\overline{G} : \overline{G}_{0}|$.   
 \end{rem}
 \begin{rem}
    \label{rem:conn_1}
    The connectedness of $\overline{G}$, i.e. whether or not $\overline{G}=\overline{G}_0$, depends just on
     the linear parts of the elements of $G$. Indeed, 
     it is equivalent to 
     $\overline{D_{\bf 0} G} = {\overline{D_{\bf 0} G}}_{0}$.
\end{rem}
 Next, we consider a particular case of this phenomenon.

 \begin{lemma}\cite[Remark 4.10]{LRSV:stable}
 	\label{lem:connected}
 Let $\phi \in \diffh{}{n}$. Then the condition
 \begin{equation}
 \label{equ:cond_con}
  \{ (m_1, \hdots, m_n) \in \zz^{n} : \prod_{k=1}^{n} \lambda_k^{m_k}  \in 
e^{2 \pi i \mathbb{Q}} \} = 
  \{ (m_1, \hdots, m_n) \in \zz^{n} : \prod_{k=1}^{n} \lambda_k^{m_k} =1 \} ,
 \end{equation}
 where
 $\mathrm{spec}(D_{\bf 0} \phi) = \{ \lambda_1, \hdots, \lambda_n\}$, is equivalent to 
 $\overline{\langle \phi \rangle} = \overline{\langle \phi \rangle}_{0}$.
 \end{lemma}
 
 \begin{defi}
        \label{def:embed}
        Let  $\phi \in \diffh{}{n}$ with  
 	$\mathrm{spec}(D_{\bf 0} \phi) = \{ \lambda_1, \hdots, \lambda_n\}$. The minimum $m \in \nn$ such that 
	$\lambda_1^{m}, \hdots, \lambda_n^{m}$ satisfy Condition  \eqref{equ:cond_con} is called the {\it index of embeddability} of $\phi$.
	We denote it by $\mm (\phi)$.
 \end{defi}
 \begin{rem}
         \label{rem:cond_con1}
        \cite[Remark 4.10]{LRSV:stable}
        Given  $\phi \in \diffh{}{n}$, $\mm (\phi)$ is the first natural number $m \in \nn$ such that $\phi^{m} \in  \overline{\langle \phi \rangle}_{0}$ or equivalently
        $\mm (\phi) = | \overline{\langle \phi \rangle}:\overline{\langle \phi \rangle}_{0}|$.
 \end{rem}
 \begin{rem}
        \label{rem:cond_con}
 	Given  $\phi \in \diffh{}{n}$  and $r \in \nn$, we have 
	$\mm (\phi^{r}) = \mm (\phi)/ \gcd (r, \mm(\phi))$ and  in particular $\mm (\phi^{\mm (\phi)}) =1$. Moreover, we have
 	$\overline{\langle \phi^{m} \rangle}  = \overline{\langle \phi^{m r} \rangle} = \overline{\langle \phi^{r} \rangle}_{0} = \overline{\langle \phi \rangle}_{0}$
	for any $r \in \mathbb{N}$ \cite[Remark 4.10]{LRSV:stable}, where $m = \mm (\phi)$.
 \end{rem}
 \subsection{Invariance properties}

 Let us consider the group of diffeomorphisms that fix an ideal $I$ of $\Oh_{n}$, or equivalently preserve its zero set $V(I)$ when 
 $I$ is generated by elements of $\Oc_{n}$.
 \begin{pro} \cite[Proposition 4.17]{LRSV:stable}
     \label{pro:inv_closed}
     The group 
     \[ \mathcal{I}_{I}= \{ \phi \in \diffh{}{n} : I \circ \phi = I \} \] 
     is Zariski-closed for any ideal $I$   of  $\Oh_{n}$.  
 \end{pro}
 The proof is based in two properties: the invariance property defines a group $\mathcal{I}_{I}$ that is also
 the solution space of a system of algebraic equations on the coefficients of the Taylor series expansion of $\phi$. 
 The next result, concerning the group of diffeomorphisms whose fixed point set contains $V(I)$,  can proved completely analogously. 
 \begin{pro}
     \label{pro:fix_closed}
     The group 
     \[ \mathcal{F}_{I}= \{ \phi \in \diffh{}{n} : f \circ \phi - f  \in I \ \forall f \in \Oh_{n}  \} \] 
     is Zariski-closed for any ideal $I$   of  $\Oh_{n}$.  
 \end{pro}
Next result makes clear how the periods for certain periodic actions associated to $\phi \in \diffh{}{n}$ are bounded by 
the index of embeddability.
\begin{cor}
     \label{cor:per_to_fix}
     Let $I$ be an ideal  of $\Oh_{n}$. Consider $\phi \in \diffh{}{n}$  
     and $r \in \mathbb{N}$ such that 
     $\phi^{r} \in \mathcal{I}_{I}$ (resp. $\mathcal{F}_{I}$). 
     Then  $\phi^{\mm (\phi)} \in \mathcal{I}_{I}$ (resp. $\mathcal{F}_{I}$).
 \end{cor}
 \begin{proof}
 Let us show the result for $\mathcal{I}_{I}$, the proof for $\mathcal{F}_{I}$ is the same. We obtain
 $\overline{\langle \phi^{m} \rangle} = \overline{\langle \phi^{m r} \rangle}$ by Remark \ref{rem:cond_con},
 where $m = \mm (\phi)$. It follows that 
 \[ \phi^{m} \in \overline{\langle \phi^{m} \rangle} = \overline{\langle \phi^{m r} \rangle} < \mathcal{I}_{I} \]
 since $\phi^{m r} \in \mathcal{I}_{I}$ and the group $\mathcal{I}_{I}$ is Zariski-closed (Proposition \ref{pro:inv_closed}).
 \end{proof}

 The Zariski-closure preserves some distinguished group classes. 
 \begin{defi}
    Let $G$ be a group. Given a property $\mathcal{P}$, we say that $G$ is {\it virtually} $\mathcal{P}$ if there exists a 
    finite index subgroup of $G$ that satisfies property $\mathcal{P}$.
\end{defi} 
 \begin{rem}\cite[Lemmas 2.5 and 2.6]{JR:finite}
     \label{rem:closure_structure}
     Let $G$ be a subgroup of $\diffh{}{n}$ such that $G$ is (virtually) abelian. Then $\overline{G}$ is  (virtually) abelian
     and $\overline{G}_0$ is abelian. The analogues obtained by replacing abelian with nilpotent or solvable also hold.
 \end{rem}

\section{The finite orbits property}
\label{sec:finite_orbits}
In this section we define the finite orbits property, first for germs of complex analytic diffeomorphisms and then for subgroups of 
$\diff{}{n}$. 
 
 \begin{defi}
 Let $\phi \in \diff{}{n}$. Consider $\epsilon \in \mathbb{R}^{+}$ such that there are representatives 
 of $\phi$ and $\phi^{-1}$, that we call also $\phi$ and $\phi^{-1}$ respectively for simplicity, defined in 
 a neighborhood of the closed ball $\overline{B} ({\bf 0}, \epsilon)$ centered at ${\bf 0}$. Let us define the positive orbit 
 $\mathcal{O}_{\phi, \epsilon}^{+}({\bf z})$ of ${\bf z} \in \overline{B} ({\bf 0}, \epsilon)$ by $\phi$ in $\overline{B} ({\bf 0}, \epsilon)$.
 We define $\varphi ({\bf z}) = \phi ({\bf z})$ if ${\bf z} \in \overline{B} ({\bf 0}, \epsilon)$ and  $\phi ({\bf z}) \in \overline{B} ({\bf 0}, \epsilon)$, 
 $\varphi ({\bf z}) = \infty$ if ${\bf z} \in \overline{B} ({\bf 0}, \epsilon)$ and  $\phi ({\bf z}) \not \in \overline{B} ({\bf 0}, \epsilon)$ and 
 $\varphi (\infty) = \infty$. Given ${\bf z} \in \overline{B} ({\bf 0}, \epsilon)$, we define 
 \[ \mathcal{I}_{\phi, \epsilon}^{+}({\bf z}) = \{0\} \cup \{ k \in \nn : \varphi^{k} ({\bf z}) \neq \infty \} . \]
 Now, we define the positive orbit
 \[ \mathcal{O}_{\phi, \epsilon}^{+}({\bf z}) = \{ \phi^{k} ({\bf z}) : k \in \mathcal{I}_{\phi, \epsilon}^{+}({\bf z}) \} . \]
 We also define $\mathcal{I}_{\phi, \epsilon}^{-}({\bf z}) = \mathcal{I}_{\phi^{-1}, \epsilon}^{+}({\bf z})$, 
 $\mathcal{I}_{\phi, \epsilon}({\bf z}) = \mathcal{I}_{\phi, \epsilon}^{+}({\bf z}) \cup \mathcal{I}_{\phi, \epsilon}^{-}({\bf z})$, 
 the negative orbit $\mathcal{O}_{\phi, \epsilon}^{-}({\bf z})  = \mathcal{O}_{\phi^{-1}, \epsilon}^{+}({\bf z})$ and the orbit 
 $\mathcal{O}_{\phi, \epsilon}({\bf z})  = \mathcal{O}_{\phi, \epsilon}^{+}({\bf z}) \cup \mathcal{O}_{\phi, \epsilon}^{-}({\bf z})$ 
 in $\overline{B} ({\bf 0}, \epsilon)$.
 \end{defi} 
 \begin{defi}
 We say that $\phi \in \diff{}{n}$ has {\it finite orbits} in $\overline{B} ({\bf 0}, \epsilon)$ if  
 $\mathcal{O}_{\phi, \epsilon}({\bf z})$ is a finite set for any ${\bf z} \in \overline{B} ({\bf 0}, \epsilon)$.
 We say that that $\phi \in \diff{}{n}$ has {\it finite orbits}, and we write $\phi \in \diff{< \infty}{n}$, if 
 $\phi \in \diff{}{n}$ has {\it finite orbits} in $\overline{B} ({\bf 0}, \epsilon)$ for some $\epsilon \in \mathbb{R}^{+}$, 
 and hence for any $\epsilon>0$ small enough.
 \end{defi}
 \begin{rem}
 \label{rem:type_finite}
 There are two types of finite orbits, namely such that $\sharp \mathcal{I}_{\phi, \epsilon}({\bf z}) < \infty$
 or $\mathcal{I}_{\phi, \epsilon}({\bf z}) =\zz$. In the latter case, there exists $k \in \zz \setminus \{0\}$ such that 
 $\phi^{k} ({\bf z}) = {\bf z}$, i.e. the orbit is periodic.
 \end{rem}
 \begin{example}
 \label{exa:fo_n}
 Every finite order element of $\diff{}{n}$ has finite orbits but the reciprocal does not hold true. 
 Indeed, let $\phi (z_1, z_2) = (z_1, z_1 + z_2) \in \diff{}{2}$. It has finite orbits since 
 $\mathcal{O}_{\phi, \epsilon}({\bf z})$ is finite for all $\epsilon \in \mathbb{R}^{+}$ and ${\bf z} \in \overline{B}(0,\epsilon)$.
 Nevertheless, it does not have finite order. 
 \end{example}
 \begin{rem}
    \label{rem:stable_manifold}
     Let $\phi \in \diff{}{n}$ that has finite orbits in $\overline{B} ({\bf 0}, \epsilon)$.
     Given  ${\bf z} \in B({\bf 0}, \epsilon)$ and  $k \in \zz$ with 
     $\mathcal{I}_{\phi, \epsilon}({\bf z})  = \zz$ and $\phi^{k} ({\bf z}) = {\bf z}$, we have 
     $\spec (D_{\bf z} \phi^{k}) \subset S^{1}$ 
     as a consequence of the stable manifold theorem
     (cf. \cite[p. 26]{Ruelle} \cite[p. 107]{Ilya-Yako}). Indeed, suppose that there exists an eigenvalue of modulus less than $1$ (if there is an
     eigenvalue of modulus greater than $1$, we consider $\phi^{-k}$). Then there exists a stable manifold $W_{\bf z}^{-}$ that is tangent to 
     \[ L_{\bf z}^{-} = \oplus_{\lambda \in \spec (D_{\bf z} \phi^{k}), \ |\lambda|<1} \ker (D_{\bf z} \phi^{k} - \lambda \mathrm{id})^{n} \]  
     at ${\bf z}$ and where all orbits are attracted to ${\bf z}$ by $\phi^{k}$. Since $W_{\bf z}^{-}$ can be constructed in any neighborhood of ${\bf z}$, 
     it follows that $\phi$ does not have finite orbits in  $\overline{B} ({\bf 0}, \epsilon)$.
 \end{rem}
 In order to define subgroups of $\diff{}{n}$ with the finite orbits property, we need to consider pseudogroups of 
 complex analytic diffeomorphisms.  We introduce such a  concept for the sake of completeness.
 \begin{defi}
     \label{def:pseudogroup}
     We say that a set $\mathcal{G}$ of holomorphic diffeomorphisms $\phi: U_{\phi} \to V_{\phi}$, where $U_{\phi}, V_{\phi}$ are open subsets 
     of a complex manifold $M$, is a pseudogroup if 
     \begin{itemize}
     \item the identity map $\mathrm{id}_{V}: V \to V$ belongs to $\mathcal{G}$ for any open subset $V$ of $M$;
     \item given $\phi: U \to V$ in $\mathcal{G}$, its inverse $\phi^{-1}:V \to U$ belongs to $\mathcal{G}$;
     \item given $\phi: U \to V$ and $\psi: V \to W$ in $\mathcal{G}$, its composition $\psi \circ \phi: U \to W$ belongs to $\mathcal{G}$;
     \item given $\phi: U \to V$ in $\mathcal{G}$, its restriction $\phi: U' \to \phi (U')$ belongs to $\mathcal{G}$ for any open subset $U'$ of $U$;
     \item given a diffeomorphism $\phi: U \to V$ such that there exists an open covering $U = \cup_{j \in J} U_j$ and 
     a family  $\{ \phi_{j}: U_j \to V_{j} : j \in J\}$ in $\mathcal{G}$ such that $\phi_{|U_j} \equiv \phi_j$ for any $j \in J$, 
     the diffeomorphism $\phi$ belongs to $\mathcal{G}$.
  \end{itemize} 
 \end{defi}
 \begin{defi}
     \label{def:germ}
     Consider a pseudogroup $\mathcal{G}$ of holomorphic diffeomorphisms defined in a complex manifold $M$.
     Given ${\bf z} \in M$, let $\mathcal{G}_{\bf z}$ be the subgroup of $\mathrm{Diff} (M, {\bf z})$ consisting of the germs
     at ${\bf z}$ of the elements $\phi: U \to V$ of $\mathcal{G}$ such that $\phi ({\bf z})= {\bf z}$.
 \end{defi}
 \begin{defi}
     \label{def:fin_orb_prop}
     Let $G$ be a subgroup of $\diff{}{n}$. 
     We say that $G$ has the {\it finite orbits} property 
     if there exists a pseudogroup $\mathcal{G}$ of holomorphic diffeomorphisms in $\cc^{n}$
     such that each orbit of $\mathcal{G}$ is a finite set and $G \subset \mathcal{G}_{\bf 0}$.
     %
 \end{defi}
  \begin{rem}
     Let $G = \langle \phi_1, \hdots, \phi_k \rangle < \diff{}{n}$ be finitely generated. 
     The finite orbits property for $G$ is equivalent to the existence of representatives
     $\phi_{j}: U_j \to V_j$, for $1 \leq j \leq k$, such that the pseudogroup generated by 
     $\{\phi_j: U_j \to V_j : 1 \leq j \leq k\}$ has finite orbits.
 \end{rem}
 \begin{rem}
      \label{rem:finite_group_indiv}
       Let $\phi  \in \diff{}{n}$. 
       The finite orbits property for $\langle \phi \rangle$ is equivalent to $\phi \in \diff{< \infty}{n}$.
       Indeed, if $\phi$ has finite orbits in $\overline{B} ({\bf 0}, \epsilon)$
       then the pseudogroup generated by any representative $\varphi: U \to V$ of $\phi$ such that 
       $U \cup V \subset B ({\bf 0}, \epsilon)$ has finite orbits. Reciprocally, if the pseudogroup generated by a representative 
       $\varphi: U \to V$ of $\phi$ has finite orbits then $\phi$ has finite orbits in $\overline{B} ({\bf 0}, \epsilon)$
       for $\overline{B} ({\bf 0}, \epsilon) \subset U \cap V$.

       Therefore, $\langle \phi \rangle$ has the finite orbits property if and only if $\langle \phi \rangle \subset \diff{< \infty}{n}$.
       In particular, a finite orbits subgroup $G$ of  $\diff{}{n}$ is contained in $ \diff{< \infty}{}$.
 \end{rem}
 \begin{example}
     Consider the subgroup $G$ of $\mathrm{GL}(1, \cc)$ of finite order elements. It can be considered as a subgroup of 
     $\diff{}{}$. 
     Clearly $G$ is contained in $\diff{< \infty}{}$ but it does not satisfy the finite orbits property.
     Anyway, any finitely generated subgroup of $G$ satisfies the finite orbits property, indeed it is a finite group. 
 \end{example}
 The case of finitely generated finite orbits subgroups of $\diff{}{n}$ is particularly simple for $n=1$, in contrast with the case
 $n \geq 2$ (Example \ref{exa:fo_n}).
  \begin{lemma}
     \label{lem:fin_orb_1}
     Let $G$ be a subgroup of $\diff{}{}$ contained in $\diff{< \infty}{}$.   
     Then $G$ consists of  finite order elements and is abelian. 
 \end{lemma}
 \begin{proof}
     A special feature of the one-dimensional theory is that  $\phi \in \diff{< \infty}{}$ if and only if its order is finite \cite[Th\'{e}or\`{e}me 2]{MaMo:Aen}. 
     Moreover, the commutators of elements of $\diff{}{}$ are tangent to the identity.
     Since finite order tangent to the identity diffeomorphisms are trivial
     (for instance because $\phi \in \overline{\langle \phi \rangle} = \overline{\langle \phi^{m} \rangle} = \{  \mathrm{id} \}$ 
     by Remark \ref{rem:cond_con}, where 
     $m = \mm (\phi)$), in particular $G$ is abelian.   
 \end{proof}
 \begin{lemma}
     \label{lem:fin_orb_1_fg}
     Let $G$ be a finitely generated subgroup of $\diff{}{}$. 
     Then $G$ has the finite orbits property if and only $G \subset \diff{< \infty}{n}$. 
     In such a case $G$ is finite and abelian. 
 \end{lemma}
 \begin{proof}
     Let us show the non-trivial implication. 
     Assume that $G \subset \diff{< \infty}{n}$.
     Then $G$ is a finitely generated abelian group of finite order elements by Lemma \ref{lem:fin_orb_1} and hence finite. 
     Clearly, it satisfies the finite orbits property.
 \end{proof}

\section{Finite orbits cyclic groups}  
\label{sec:cyclic}
Our goal is proving that the torsion locus of a finite orbits subgroup of $\diff{}{n}$ is non-trivial 
(Proposition \ref{pro:cc_is_periodic}) and Corollary \ref{cor:periodic_set}).  
This section considers the case of cyclic groups, the general case will be treated in next section. 

Our approach is topological.
We profit from the finite orbits property to build a non-trivial continuum $\mathcal{K}$ of periodic orbits (Proposition \ref{pro:exists_continuum}). 
Moreover, the non-hyperbolic nature of eigenvalues at periodic points  (Remark \ref{rem:stable_manifold}) and the reduction of periods
result provided by Corollary \ref{cor:per_to_fix}, force $\mathcal{K}$ to be analytic in a neighborhood of the origin (Proposition \ref{pro:cc_is_periodic}).
Let us remark that the results in this section can be considered as a generalization of the main theorem of \cite{Lisboa-Ribon:fixed-point}, where
it is considered the case $n=2$.
 
First, we introduce a  continuum of periodic orbits.  Then we will see that it is non-trivial.
  
 \begin{defi}
     Let $\phi \in \diff{}{n}$ and $\epsilon \in \mathbb{R}^{+}$. We define
     \[ \mathcal{K}_{\phi, \epsilon}' = \{ {\bf z} \in \overline{B}({\bf 0},\epsilon) : \mathcal{I}_{\phi, \epsilon}({\bf z}) = \zz \} \]
     and $\mathcal{K}_{\phi, \epsilon}$ as the connected component of $\mathcal{K}_{\phi, \epsilon}'$ containing ${\bf 0}$.
 \end{defi}
 \begin{rem}
 Both $\mathcal{K}_{\phi, \epsilon}'$ and $\mathcal{K}_{\phi, \epsilon}$ are compact sets.
 \end{rem}
 \begin{pro}
     \label{pro:exists_continuum}
     Let $\phi \in \diff{<\infty}{n}$ and $W$ a germ of $\phi$-invariant analytic variety at ${\bf 0}$ of positive dimension. 
     Then 
     $\mathcal{K}_{\phi, \epsilon}$
     contains a continuum $K_{W, \epsilon}$ 
     such that ${\bf 0} \in K_{W, \epsilon} \subset W$ and $K_{W, \epsilon} \cap \partial B({\bf 0}, \epsilon) \neq \emptyset$
     for any $\epsilon \in \mathbb{R}^{+}$ small enough. In particular  $\mathcal{K}_{\phi, \epsilon}$ is a non-trivial continuum.
 \end{pro}
 \begin{proof}
 The set $W_{\epsilon} := W \cap \overline{B} ({\bf 0}, \epsilon)$ is connected if 
$\epsilon \in \mathbb{R}^{+}$ is sufficiently small by the local conic structure 
of analytic sets \cite[Lemma 3.2]{Burghelea-Verona:local_analytic} \cite[Theorem 2.10]{Milnor:hypersurfaces}.
 It suffices to show that there exists a continuum $K$ such that ${\bf 0} \in K \subset W \cap \overline{B}({\bf 0}, \epsilon)$, 
 $K \cap \partial B({\bf 0}, \epsilon) \neq \emptyset$
 and $\phi (K) \subset K$. Indeed, the last property implies $\mathcal{I}_{\phi, \epsilon}^{+}({\bf z}) = \zz_{\geq 0}$
 for any ${\bf z} \in K$ and thus $K$ consists of periodic orbits. 
 In particular, we obtain $K \subset \mathcal{K}_{\phi, \epsilon}'$ and hence $K \subset \mathcal{K}_{\phi, \epsilon} \cap W$. 
 
 First, suppose that there exists $\delta \in (0,\epsilon]$  
such that $\phi^{-k} (W_{\delta}) \subset \overline{B}({\bf 0}, \epsilon)$ for any $k \in \mathbb{N}$.
Since $\phi \in \diff{<\infty}{n}$, the orbit of ${\bf z}$ is periodic for any ${\bf z} \in W_{\delta}$.
Consider the compact set $W_{\delta, k} = \{ {\bf z} \in W_{\delta} : \phi^{k}({\bf z}) = {\bf z} \}$ for $k \in \mathbb{N}$.
Given an irreducible component $W_0$ of $W$, there exists $W_{\delta, k_0}$ such that 
it has a non-empty interior in $W_0$ by Baire's category theorem and hence $W_{0} \cap W_{\delta} = W_{\delta, k_0}$
by the identity theorem. By applying this argument to every irreducible component of $W$, we 
deduce that there exists $\kappa \in \mathbb{N}$ such that 
$W_{\delta} = W_{\delta, \kappa}$.
Another application of the identity principle provides that $W_{\delta'} = W_{\delta', \kappa}$
for any $\delta' \in [\delta, \epsilon]$ such that 
$W_{\delta'} \subset \mathcal{K}_{\phi, \epsilon}'$. 
Consider the supremum (maximum) $\delta''$ of such $\delta' \in [\delta, \epsilon]$ and define 
the set $K = \cup_{k=0}^{\kappa -1} W_{\delta''}$.  We obtain 
$K \subset \mathcal{K}_{\phi, \epsilon}' \subset  \overline{B}({\bf 0}, \epsilon)$ by construction. 
Moreover $K$ intersects $\partial B({\bf 0}, \epsilon)$ in a non-empty set  by the choice of $\delta''$ since otherwise we would get 
$W_{\delta'} \subset \mathcal{K}_{\phi, \epsilon}'$ for any $\delta'$ slightly bigger than $\delta''$.
We are done since by construction, $K$ is a continuum  such  that ${\bf 0} \in K \subset W$ and 
$K \subset \mathcal{K}_{\phi, \epsilon}'$.
 
Let us consider the remaining case. 
Fix $\delta \in (0,\epsilon]$.  Let $k$ be the minimum natural number such that 
 $\phi^{-k} (W_{\delta}) \setminus \overline{B}({\bf 0}, \epsilon) \neq \emptyset$. 
Then, up to replace $\delta$ with a smaller positive real number, we can suppose that 
$\tilde{W}_{\delta} := \cup_{0 \leq j \leq k} \phi^{-j}(W_{\delta})$ is contained in $\overline{B}({\bf 0}, \epsilon)$ and 
$\tilde{W}_{\delta} \cap \partial {B}({\bf 0}, \epsilon) \neq \emptyset$.
The set $\tilde{W}_{\delta}$ is compact and moreover is connected since it is a union of connected sets with the common point ${\bf 0}$.
Moreover, we obtain
\[ \phi (\tilde{W}_{\delta}) \subset \tilde{W}_{\delta} \cup \phi (\overline{B} ({\bf 0}, \delta)) \]
by construction. Following this idea, consider a sequence ${\{\tilde{W}_{\delta_m}\}}_{m \in \mathbb{N}}$ where 
${(\delta_m)}_{m \geq 1}$ converges to $0$.
Since the space of continua contained in $\overline{B} ({\bf 0}, \epsilon)$ is compact in the Hausdorff topology,  there exists a subsequence
of ${\{\tilde{W}_{\delta_m}\}}_{m \in \mathbb{N}}$ that converges to a continuum $K$ such that 
${\bf 0} \in K \subset W \cap \overline{B}({\bf 0}, \epsilon)$, $K \cap \partial B({\bf 0}, \epsilon) \neq \emptyset$ and $\phi (K) \subset K$.
 \end{proof}
 
 \begin{rem}
     The ideia of considering continua associated to finite orbits germs of biholomorphisms is already in the work of Mattei and Moussu
     \cite{MaMo:Aen}, where it is applied in dimension $1$. Similar constructions in higher dimension can be also found in 
     \cite{RR:arxiv} and \cite{Lisboa-Ribon:fixed-point}.
 \end{rem}
 
 The following result is a consequence of applying Proposition \ref{pro:exists_continuum} to $W = (\cc^{n}, {\bf 0})$.
  
\begin{cor}
     Let $\phi \in \diff{<\infty}{n}$. 
     Then $\mathcal{K}_{\phi, \epsilon}$ is a non-trivial continuum  for any $\epsilon \in \mathbb{R}^{+}$ small enough. 
 \end{cor}

 Our next goal is showing that $\mathcal{K}_{\phi, \epsilon}$ is semi-analytic and, moreover, it is complex analytic in a neighborhood of 
 ${\bf 0}$.
 
 \begin{defi}
     Let $\phi \in \diff{<\infty}{n}$ and $\epsilon \in \mathbb{R}^{+}$ small. We define
     \[ \overline{\mathrm{Per}}_{\phi, \epsilon, k} = \{ {\bf z} \in \mathcal{K}_{\phi, \epsilon}' : \phi^{k} ({\bf z}) = {\bf z} \}  \]
     for $k \in \mathbb{N}$.
 \end{defi}

 \begin{rem}
      Given $\phi \in \diff{<\infty}{n}$, we have $\mathcal{K}_{\phi, \epsilon}' = \cup_{k \in \mathbb{N}} \overline{\mathrm{Per}}_{\phi, \epsilon, k}$. 
 \end{rem}
\begin{lemma}
    \cite[Lemma 4]{Lisboa-Ribon:fixed-point}
    Let $\phi \in \diff{<\infty}{n}$ and $\epsilon \in \mathbb{R}^{+}$ small. Then the set $\overline{\mathrm{Per}}_{\phi, \epsilon, k}$
    is semianalytic and has finitely many connected components that are all compact, path connected and semianalytic  for any $k \in \mathbb{N}$.
\end{lemma}
\begin{defi}
    \label{def:chain}
    Consider $\epsilon'$ slightly bigger than $\epsilon$ and the analytic set
    \[ \mathcal{V}_{k} = \{ {\bf z} \in B(0, \epsilon') : \{ \phi(z), \hdots, \phi^{k}(z)\} \subset B(0, \epsilon') \ \mathrm{and} \ \phi^{k}({\bf z}) = {\bf z} \} . \]
    for $k \in \mathbb{N}$. This set contains $\overline{\mathrm{Per}}_{\phi, \epsilon, k}$.
    Consider the connected components $\mathcal{C}_{1}^{k}, \hdots, \mathcal{C}_{j_k}^{k}$ of $\overline{\mathrm{Per}}_{\phi, \epsilon, k}$
    for $k \in \mathbb{N}$. We define the set 
    \[ \mathcal{A} = \{ \mathcal{C}_{j}^{k} : k \in \mathbb{N}, \ 1 \leq j \leq j_k \} \]
    of connected components of periodic points.  
    We say that $\mathcal{C}_{j_1}^{k_1}, \hdots, \mathcal{C}_{j_l}^{k_l}$ is a {\it chain} in $\mathcal{A}$ if 
    $\mathcal{C}_{j_r}^{k_r} \cap  \mathcal{C}_{j_{r+1}}^{k_{r+1}} \neq \emptyset$ for any $1 \leq r < l$.
    The relation ``belonging to a chain" is an equivalence relation in $\mathcal{A}$.
\end{defi}  
\begin{pro}
\label{pro:cc_is_periodic}
Let $\phi \in \diff{<\infty}{n}$ and $\epsilon \in \mathbb{R}^{+}$ small. Then there exists $k \in \mathbb{N}$ such that 
$\mathcal{K}_{\phi, \epsilon}$ is a connected component of $\overline{\mathrm{Per}}_{\phi, \epsilon, \mm (\phi)}$. 
\end{pro}
\begin{proof}
Consider the setting in Definition \ref{def:chain}. Let ${\bf z}_1 \in \mathcal{K}_{\phi, \epsilon}' \cap  \overline{\mathrm{Per}}_{\phi, \epsilon, a}$. 
Denote $k_1 =  \mm ({\bf z}_1) a$ and $\psi = \phi^{k_{1}}$, where
$\mm ({\bf z}_1)$ is the embeddability index of the germ $\phi_{{\bf z}_{1}}^{a}$ at ${\bf z}_{1}$. Let  $ \mathcal{C}_{j_1}^{k_1}$ 
be the connected component of $\overline{\mathrm{Per}}_{\phi, \epsilon, k_1}$ containing ${\bf z}_1$. 
Let $\mathcal{C}_{j_1}^{k_1}, \hdots, \mathcal{C}_{j_l}^{k_l}$ be a chain in $\mathcal{A}$. 
Denote $s = k_1 \hdots k_l$, $\rho = \phi^{s}$ and 
$\mathcal{C}:= \mathcal{C}_{j_1}^{k_1} \cup \hdots \cup \mathcal{C}_{j_l}^{k_l}$. Note that $\mathcal{C}$ is contained in 
$\fix (\rho)$.  Our goal is proving that $\mathcal{C}= \mathcal{C}_{j_1}^{k_1}$.
We define 
\[ F = \{ {\bf z} \in \mathcal{C}  : \psi ({\bf z}) = {\bf z}  \ \mathrm{and} \ \spec (D_{\bf z} \psi) = \spec (D_{{\bf z}_{1}} \psi)  \} , \]
where eigenvalues are considered with multiplicity, i.e. $\spec (D_{\bf z} \psi) = \spec (D_{{\bf z}_{1}} \psi)$ is equivalent to the equality of the characteristic polynomials of 
$D_{\bf z} \psi$ and $D_{{\bf z}_{1}} \psi$. The set $F$ is clearly closed by the continuity of eigenvalues. 

We claim that $F$ is open. 
Given ${\bf z}_0 \in F$ arbitrary, we have  $\rho_{ {\bf z}_0} \in \mathcal{F}_{I(\mathcal{V}_{s}, {\bf z}_0)  }$
(cf. Proposition \ref{pro:fix_closed}), where $I(\mathcal{V}_{s}, {\bf z}_0)$ is the ideal of 
$\mathcal{V}_{s}$ in the local ring $\mathbb{C} \{ {\bf z} - {\bf z}_{0} \}$.
Since $\mm (\psi_{{\bf z}_{1}}) =1$ by Remark \ref{rem:cond_con} and 
$ \spec (D_{{\bf z}_0} \psi) = \spec (D_{{\bf z}_{1}} \psi)$, it follows that $\mm (\psi_{{\bf z}_{0}}) =1$
by Lemma \ref{lem:connected}.
We get  $\psi_{ {\bf z}_0} \in \mathcal{F}_{I(\mathcal{V}_{s}, {\bf z}_0) }$ and hence
$(\mathcal{C}, {\bf z}_0) \subset (\fix (\psi) , {\bf z}_0)$ by Corollary  \ref{cor:per_to_fix}.
Since  $\psi_{\bf z} \in \mathrm{Diff}_{<\infty} (\cc^{n}, {\bf z})$, we deduce $\spec (D_{\bf z} \psi) \subset S^{1}$ by 
Remark \ref{rem:stable_manifold}
for any ${\bf z} \in \mathcal{V}_{s}$ in a neighborhood of ${\bf z}_{0}$. Since eigenvalues of $\spec (D_{{\bf z}'} \psi)$ vary analytically 
with respect to ${\bf z}' \in \mathcal{V}_{s}$, it follows that $\spec (D_{{\bf z}'} \psi) = \spec (D_{{\bf z}} \psi)$
for ${\bf z}'$ in a neighborhood of ${\bf z}$ in $\mathcal{V}_{s}$ by the open mapping theorem. 
In particular, we obtain that ${\bf z}_{0}$ belongs to the interior of $F$. 
It follows that $F$ is open. Moreover, $F$ is equal to $\mathcal{C}$ since ${\bf z}_1 \in F$ and the latter set is connected.
This implies $\mathcal{C} \subset \fix (\psi)$. 

Consider the union $\mathcal{D}$ of all elements of $\mathcal{A}$ in the same equivalence class as $\mathcal{C}_{j_1}^{k_1}$. 
Since $\mathcal{D}$ is connected by construction, the
previous discussion implies $\mathcal{D} =\mathcal{C}_{j_1}^{k_1}$.  
The construction  and ${\bf 0} \in \mathcal{K}_{\phi, \epsilon}' \cap  \overline{\mathrm{Per}}_{\phi, \epsilon, 1}$ imply that
the union of the sets in the chain containing ${\bf 0}$ is equal to an element 
$\mathcal{C}_{j_{0}}^{1 \cdot \mm ({\bf 0})} = \mathcal{C}_{j_{0}}^{m} $ of $\mathcal{A}$ and hence
contained in $\overline{\mathrm{Per}}_{\phi, \epsilon, m}$, where $m = \mm(\phi)$.

Since equivalence classes in $\mathcal{A}$ are disjoint, we can express the continuum $\mathcal{K}_{\phi, \epsilon}$ as a countable union of disjoint 
closed sets of the form $\mathcal{K}_{\phi, \epsilon} \cap \mathcal{C}_{j}^{k}$. 
Sierpi\'{n}ski's theorem (cf. \cite[p.358]{Engelking:general_topology_89}) implies that $\mathcal{K}_{\phi, \epsilon}$ is contained in one of them and hence 
$\mathcal{K}_{\phi, \epsilon} = \mathcal{C}_{j_{0}}^{m} \subset \overline{\mathrm{Per}}_{\phi, \epsilon, m}$.  
\end{proof}
Next, we interpret the previous result in terms of the torsion locus. We see that it is non-trivial if 
$\phi \in \diff{<\infty}{n}$.
\begin{defi}
    \label{def:fg}
    Let $G$ be a subgroup of $\diff{}{n}$.  We define the {\it torsion locus} $\fg$ of $G$ as the germ of analytic variety 
    $\fg = \cap_{\phi \in G, \, \mm (\phi)=1} \fix (\phi)$. The same expression defines $\fg$ for 
    $G < \diffh{}{n}$.
\end{defi}
\begin{rem}
    The torsion locus can be also expressed as $\fg = \cap_{\phi \in G} \fix (\phi^{\mm (\phi)})$
    by Remark \ref{rem:cond_con}.
    We have $\mathfrak{T}_{\langle \phi \rangle} = \fix (\phi^{\mm (\phi)})$ for $\phi \in \diffh{}{n}$, again by Remark \ref{rem:cond_con}.
\end{rem}

\begin{cor}
    \label{cor:pos_inter}
    Let $\phi \in \diff{<\infty}{n}$. Consider a germ of $\phi$-invariant analytic set $W$ at ${\bf 0}$ of positive dimension.
    Then $\dim (\mathfrak{T}_{\langle \phi \rangle} \cap W) \geq 1$ holds.
\end{cor}
\begin{proof}
We can suppose that $W$ has pure dimension without loss of generality.
Consider a small $\epsilon \in \mathbb{R}^{+}$. 
There exists a non-trivial continuum $K_{W,\epsilon}$ contained in $\mathcal{K}_{\phi, \epsilon} \cap W$ and containing ${\bf 0}$
by Proposition \ref{pro:exists_continuum}. Moreover, 
$\mathcal{K}_{\phi, \epsilon} \subset  \overline{\mathrm{Per}}_{\phi, \epsilon, m}$ holds by Proposition \ref{pro:cc_is_periodic}, 
where $m = \mm (\phi)$.
We obtain 
\[ (K_{W, \epsilon}, {\bf 0}) \subset (\mathrm{Fix}(\phi^{m}) \cap W,{\bf 0}) =  (\mathfrak{T}_{\langle \phi \rangle}  \cap W,{\bf 0})  \]  
and, as a consequence, 
the latter germ (of analytic set) has positive dimension.
\end{proof}

\begin{cor}
    \label{cor:periodic_set}
    Let $\phi \in \diff{<\infty}{n}$. Then $\dim (\mathfrak{T}_{\langle \phi \rangle} ) \geq 1$ holds. 
\end{cor}  
    Since  $\dim (\mathrm{Fix} (\phi^{m}), {\bf 0}) \geq 1$ implies $1 \in  \spec (D_{\bf 0} \phi^{m})$, we obtain:
\begin{cor}
    \label{cor:eigen}
    Let $\phi \in \diff{<\infty}{n}$. Then $\spec (D_{\bf 0} \phi)$ contains a root of unity.  
\end{cor} 
    Corollary \ref{cor:eigen} and Remark \ref{rem:finite_group_indiv} immediately imply Corollary \ref{cor:main_eigen}.
\begin{rem}
    \label{rem:just_1_eigen}
    Corollary \ref{cor:eigen} is sharp since there are finite orbits germs of biholomorphism in $\diff{}{n}$, for $n \geq 2$, 
    such that the number of eigenvalues (counted with multiplicity) of $D_{\bf 0} \phi$
    that are roots of unity is equal to $1$
    \cite[Theorem 1]{Lisboa-Ribon:fixed-point}.
\end{rem}

\section{General finite orbits groups} 
\label{sec:general}
We show in this section Theorems \ref{thm:main_g} and \ref{thm:main_g_fg}. 

\subsection{The torsion locus for general groups}
First, we introduce analogues  for the case where we do not require the group $G$ to have 
finite orbits. Obviously such a property will be required to guarantee that $\fg$ has positive dimension.

\begin{defi}
    \label{def:self}
    Given a subgroup $G$ of self-maps, we define $\fix (G) = \cap_{\phi \in G} \fix (\phi)$.
\end{defi}
\begin{pro}
    \label{pro:fg}
    Let $G$ be a subgroup of $\diffh{}{n}$.
    Then $\fg$ is $G$-invariant and $G_{|\fg}$ is a virtually abelian group consisting of finite order elements.
    Moreover $\fg$ is maximal among the formal varieties $V$
    such that $V$ is $G$-invariant and 
    $G_{|V}$ consists of finite order elements.
\end{pro}
\begin{proof}
We define $G_{0}^{\natural} = \langle \overline{\langle \phi \rangle}_{0} : \phi \in G \rangle = \langle \overline{\langle \phi \rangle} : \phi \in G, \, \mm(\phi)=1 \rangle$ 
(Remark \ref{rem:cond_con}) and its closure $G_{0}^{\sharp}$
in the Krull topology (cf. Definition \ref{def:krull}).
The Main Theorem of \cite{Ribon:Zariski-closure} states that, even if $G_{0}^{\sharp}$ is not necessarily the Zariski-closure of $G$, it is 
not that far away since  $G_{0}^{\sharp}$ is a pro-algebraic normal subgroup of $\overline{G}$
whose codimension in $\overline{G}$ is finite. More precisely, 
there exists a short exact sequence such that 
\begin{equation}
\label{equ:exact}
 0 \to G_{0}^{\sharp} \xhookrightarrow{}   \overline{G} \stackrel{ \tau_1\circ \hat{\pi}_1}{\longrightarrow} \overline{G}_{1}/  (G_{0}^{\sharp})_1 \to 0 
 \end{equation}
and $\overline{G}_{1}/  (G_{0}^{\sharp})_1$ is isomorphic to a virtually abelian matrix d-group (i.e. consisting only of diagonalizable elements).
This exact sequence is the main ingredient to show that  $G_{|\fg}$ is virtually abelian.

The definitions of $G_{0}^{\sharp}$ and $G_{0}^{\natural}$ imply 
$\fix (G_{0}^{\sharp}) \subset \fix (G_{0}^{\natural}) \subset \fg$.
Consider the ideal $I = I(\fg)$ of $\Oh_{n}$.
Since $\phi \in \mathcal{F}_{I}$ for any $\phi \in G$ with $\mm (\phi)=1$  by Corollary \ref{cor:per_to_fix}, 
and $\mathcal{F}_{I}$ is a Zariski-closed group, we deduce
that  $G_{0}^{\natural} < \mathcal{F}_{I}$.
In particular, we obtain  $G_{0}^{\sharp} < \mathcal{F}_{I}$ since every Zariski-closed group is Krull-closed 
(Remark \ref{rem:zar_imp_krull}).
Therefore $\fg$ is contained in $\fix (G_{0}^{\sharp})$ and thus $\fg = \fix (G_{0}^{\sharp})$.
Since $G_{0}^{\sharp}$ is a Zariski-closed normal subgroup of $\overline{G}$, it follows that
$\fix (G_{0}^{\sharp})$ is $\overline{G}$-invariant. In particular $\fg$ is $G$-invariant.

Consider a formal variety $V$, or equivalently an ideal $J$ of $\Oh_{n}$, such that $V$ is $G$-invariant and 
$G_{|V}$ consists of finite order elements.
Given $\phi \in G$,  
there exists $r \in \nn$ such that $\phi^{r} \in \mathcal{F}_{J}$ by hypothesis.
We obtain $\phi^{\mm (\phi)} \in \mathcal{F}_{J}$ by Corollary \ref{cor:per_to_fix}. 
Therefore $V$ is contained in $\fg = \cap_{\phi \in G, \, \mm (\phi)=1} \fix (\phi)$.

By the exact sequence \eqref{equ:exact} and since the action of $\overline{G}$ on $\fg$ factors through $\overline{G}/ G_{0}^{\sharp}$, 
we deduce that $G_{|\fg}$ is  virtually abelian.
Finally $G_{|\fg}$ consists of finite order elements since 
$\phi^{\mm (\phi)} \in G_{0}^{\natural}$ for any $\phi \in G$.
\end{proof}
We can be much more precise in the finitely generated case. 

\begin{pro}
    \label{pro:fgfg}
     Let $G$ be a finitely generated subgroup of $\diffh{}{n}$.
     Then $\fg$ is $G$-invariant, $G_{|\fg}$ is a finite group and  
     $\fg = \fix (\overline{G}_{0} \cap G)$. 
\end{pro}
\begin{proof}
The formal variety $\fg$ is $G$-invariant and $G_{|\fg}$ is virtually abelian by Proposition \ref{pro:fg}.
Let $J$ be an  finite index normal  abelian  subgroup of $G_{|\fg}$; it is finitely generated.
Since $J$ is a finitely generated abelian group of finite order elements, $J$ is finite. Therefore, 
$G_{|\fg}$ is finite.

Denote $\mathfrak{T} = \fix (\overline{G}_{0} \cap G)$.
We obtain  $\mathfrak{T} \subset \fg$ since 
\[ \phi \in \overline{\langle \phi \rangle} = {\overline{\langle \phi \rangle}}_0 < \overline{G}_0 \]
for any $\phi \in G$ with $\mm (\phi)=1$.
Denote $I = I(\fg)$. 
Since $G_{|\fg}$ is finite, the group $\mathcal{F}_{I} \cap G$ is a finite index normal subgroup of $G$.
As, a consequence the Zariski-closure $\overline{\mathcal{F}_{I} \cap G}$ is a 
finite index normal subgroup of $\overline{G}$
that contains $\overline{G}_0$ by Remark \ref{rem:fin_ind}. In particular, $\overline{G}_0$ is a subgroup of 
$\overline{\mathcal{F}_{I}} = \mathcal{F}_{I}$ since $\mathcal{F}_{I}$
is pro-algebraic (Proposition \ref{pro:fix_closed}). This implies $\fg \subset \mathfrak{T}$ and thus $\mathfrak{T} = \fg$.
\end{proof}

\begin{rem}
    Proposition \ref{pro:fg} implies that if $V$ is a a formal variety that is $G$-invariant, for 
    some $G < \diffh{}{n}$, and $G_{|V}$ consists of finite order elements then  $G_{|V}$ is virtually abelian.
    It is finite if $G$ is finitely generated by Proposition \ref{pro:fgfg}.
\end{rem}

 \subsection{The torsion locus for finite orbits groups}
 Now, that we already considered the general case, let us focus on the case of groups with finite orbits. 
 The proof of Theorem \ref{thm:main_g} is based on the following points: 
 $G$ has a strong algebraic structure (cf. Proposition \ref{pro:virt_solv}), 
 the results on cyclic groups in section \ref{sec:cyclic} and the theory of pro-algebraic subgroups of $\diffh{}{n}$.
 The first two ingredients are of topological nature, whereas the last one is algebraic. 
 
 
\begin{defi}
    Let $G$ be a group. We define the $0$-th derived subgroup $G^{0} =G$ of $G$ and then, recursively, 
    the $l$-{\it th derived subgroup} 
    \[ G^{l} = [G^{l-1}, G^{l-1}] = \langle [\phi, \psi] : \phi, \psi \in G^{l-1} \rangle \] 
    for any $l \in \zz_{\geq 1}$, where $[\phi, \psi] = \phi \psi \phi^{-1} \psi^{-1}$ is the commutator of $\phi$ and $\psi$.
    The group $G^{1}$ is called the {\it derived group} of $G$ and is also denoted by $G'$.
    We say that $G$ is {\it solvable} if there exists $l \in \zz_{\geq 0}$ such that $G^{l} = \{1\}$.
    In such a case, the minimum $l \in  \zz_{\geq 0}$ such that $G^{l} = \{1\}$ is called the 
    {\it derived length} of $G$. 
\end{defi} 
\begin{pro}
    \label{pro:virt_solv}
     Let $G$ be a subgroup of $\diff{}{n}$ that has the finite orbits property. Then $G$ is virtually solvable.
\end{pro}
\begin{proof}
Aiming at a contradiction, assume that $G$ is not virtually solvable. This property motivates us to 
apply a version of Zassenhauss Lemma to obtain commutators of elements of $G$ that 
are as close to identity as desired (cf. \cite{Ghys-identite}).

We have $G \subset \diff{<\infty}{n}$ by Remark \ref{rem:finite_group_indiv}.
This implies $\spec (D_{\bf 0} \phi) \subset S^{1}$ for any $\phi \in G$ by Remark \ref{rem:stable_manifold}.
This property plus $G$ being non-virtually solvable allow to apply 
Theorem 1 (case 2) in \cite{JR:solvable3rec}, the result providing the commutators close to identity.
Indeed, there exists a finite subset $\mathcal{S} = \{ \phi_1, \hdots, \phi_k \}$ of $G$ such that 
the pseudogroup $\mathcal{G}$ generated by 
\[ \{ \phi_{j}: U_j \to V_j : 1 \leq j \leq k\} \] 
has infinite orbits for any  choice  of representatives  $\phi_{j}: U_j \to V_j$  for $1 \leq j \leq k$.
More precisely, there exists a connected open neighborhood $V$ of ${\bf 0}$ and a sequence
${\{ \psi_n \}}_{n \geq 1}$ in $\mathcal{G} \setminus \{ \mathrm{id}_V \}$ of diffeomorphisms defined in $V$ and such that 
${\{ \psi_n \}}_{n \geq 1}$ converges uniformly to $\mathrm{id}_V$ in $V$. Thus, the generic $\mathcal{G}$-orbit of a point in 
$V$ is recurrent and hence infinite, contradicting the hypothesis.
\end{proof}

Let $G$ be a finite orbits subgroup of $\diff{}{n}$. 
In order to show Theorem \ref{thm:main_g}, we build a germ of analytic variety such that $V$ is $G$-invariant, $G_{|V}$
consists of finite order elements and $\dim (V) \geq 1$. Since $V \subset \fg$, it follows that $\fg$ is non-trivial.
Such a $V$ will be constructed as an intersection of fixed point sets
of infinite order elements, whose existence is guaranteed by next lemma. 
 
\begin{lemma}
    \label{lem:exist_infinite}
     Let $G$ be a solvable subgroup of $\diff{}{n}$ that has finite orbits and satisfies $\overline{G} = \overline{G}_0$.
     Assume that $W$ is a germ of analytic variety at ${\bf 0}$ that is $G$-invariant. 
     Denote by $\ell$ the derived length of $G_{|W}$ and suppose $\ell \geq 2$. Then, there are elements of infinite order in 
     $G_{|W}^{\ell -1}$.
\end{lemma}
Note that $\ell$ is the minimum $l \in \zz_{\geq 0}$ such that $W$ is contained in $\fix (G^{l})$.
\begin{proof}
Given a subgroup $H$ of $\diffh{}{n}$, the condition $\overline{H} = \overline{H}_0$ is equivalent
to $\overline{D_{\bf 0} H} = \overline{D_{\bf 0} H}_{0}$ (Remark \ref{rem:conn_1}).  
We deduce that $\overline{D_{\bf 0} G}$ is a connected algebraic matrix group. 
Since $G$ is solvable, $D_{\bf 0} G$ is also solvable. 
Therefore $\overline{D_{\bf 0} G}$ is solvable \cite[section I.2.4, Corollary 1]{Borel}. 
Thus, Lie-Kolchin's Theorem (cf. \cite[section 17.6, p. 113]{Humphreys})
implies that there exists a linear change of coordinates such that all elements of  $D_{\bf 0} G$ are
upper triangular. Since $\ell -1 \geq 1$, we obtain that every element of $G^{\ell -1}$ is unipotent.

Let $\phi \in G^{\ell -1}$ such that $\phi_{|W}$ has finite order. 
In particular, there exists $r \in \nn$ such that 
$\phi^{r} \in \mathcal{F}_{I(W) }$. Since the index of embeddability of $\phi$ is equal to $1$ 
(cf. Lemma \ref{lem:connected}), it follows that  
$(W, {\bf 0}) \subset \fix (\phi)$ by Corollary \ref{cor:per_to_fix}. So there is an element of infinite order in $G_{|W}^{\ell -1}$
since otherwise  $G_{|W}^{\ell -1}$ is trivial and hence the derived length of 
$G_{|W}$ is at most $\ell -1$.
\end{proof}

Let $G$ be a finite orbits subgroup of $\diff{}{n}$. Assume that there exists a germ $V$ of analytic variety at ${\bf 0}$
that is $G$-invariant. In order to prove Theorem \ref{thm:main_g}, there are two cases, either $G_{|V}$ consists of finite order 
elements and we are done, or the next result provides another invariant analytic variety whose dimension is positive but smaller
than $\dim (V)$, so the problem is simplified.

\begin{lemma}
    \label{lem:red_dim}
    Let $G$ be a solvable subgroup of $\diff{}{n}$ that has finite orbits.
    Assume that $W$ is an irreducible  germ of analytic variety at ${\bf 0}$ that is of positive dimension and $G$-invariant. 
    Denote by $\ell$ the derived length of $G_{|W}$ and assume $\ell \geq 1$. Consider the sets
    \[ \mathcal{A}_{G, W} = \{ \phi \in G^{\ell -1}:  \phi_{|W} \ {\it has \ infinite \ order} \} \ \mathrm{and} \  
    \mathcal{A}_{G,W}^{0} = \{ \phi \in \mathcal{A}_{G, W}  :  \mm (\phi) =1 \} . \]
    Suppose $ \mathcal{A}_{G, W} \neq \emptyset$. Then $\fix (\mathcal{A}_{G,W}^{0}) \cap W$ is a $G$-invariant germ of analytic set such that 
    $1 \leq \dim (\fix (\mathcal{A}_{G,W}^{0}) \cap W) < \dim (W)$.
\end{lemma}
\begin{proof}
Denote $\mathcal{A} = \mathcal{A}_{G, W}$ and $\mathcal{A}_{0} = \mathcal{A}_{G, W}^{0}$.
Note that $\phi \in \mathcal{A}$ implies $\phi^{\mm (\phi)} \in \mathcal{A}_{0}$ 
by Remark \ref{rem:cond_con} and hence $ \mathcal{A}_{0} \neq \emptyset$. 
Since $G^{\ell -1} \lhd G$, it follows that $\phi \mathcal{A} \phi^{-1} = \mathcal{A}$ and thus
$\phi \mathcal{A}_0 \phi^{-1} = \mathcal{A}_0$ for any $\phi \in G$. Therefore, 
$\fix (\mathcal{A}_{0}) \cap W$ is $G$-invariant.

By noetherianity of $\Oc_{n}$, 
there exist $\phi_1, \hdots, \phi_r \in  \mathcal{A}_{0}$ such that the germs 
$\fix (\mathcal{A}_{0})$ and $\fix (\phi_1, \hdots, \phi_r)$ at ${\bf 0}$ coincide.
Denote $F_k = W \cap \cap_{1 \leq j \leq r} \fix (\phi_j)$ for $0 \leq k \leq r$.  We have $F_0 = W$, $F_{1} \neq W$ and 
$F_{k+1} \subset F_k$ for any $0 \leq k < r$. It suffices to show 
\begin{equation}
\label{equ:fk_w}
 1 \leq \dim (F_k \cap W) < \dim (W) 
\end{equation}
for $1 \leq k \leq r$ by induction on $k$. The result for $k=1$ is a consequence of Corollary \ref{cor:pos_inter} and $F_{1} \neq W$.
Assume the result is proved for $1 \leq k < r$. 
Since  $G_{|W}^{\ell -1}$ is abelian, $F_{k} \cap W$ is $\phi_{k+1}$-invariant and hence Equation \eqref{equ:fk_w}
derives from Corollary \ref{cor:pos_inter} and $F_{k+1} \subset F_k$.
\end{proof}

Finally, we show the existence of a germ of analytic variety $V$  that
is $G$-invariant and $G_{|V}$ consists of finite order elements.

\begin{pro}
    \label{pro:main}
    Let $G$ be a subgroup of $\diff{}{n}$ that has the finite orbits property. Then $G$ is virtually solvable and 
    there exist a finite index normal subgroup $H$ of $G$ and a  germ of analytic variety $V$ at ${\bf 0}$ of positive dimension
    such that $V$ is $G$-invariant, $G_{|V}$  consists of finite order elements and  $H_{|V}$ is abelian.
\end{pro}
\begin{proof} 
The group $G$ is virtually solvable by Proposition \ref{pro:virt_solv}.
Assume, without loss of generality,  that $V$ is a germ  of analytic variety at ${\bf 0}$ of positive dimension such that 
$V$ is the $G$-orbit of any of its irreducible components. We define $H = J \cap {G}_{0}$, 
where ${G}_{0} = G \cap \overline{G}_{0}$ and
\[ J = \{ \phi \in G: \phi (W) = W \ {\rm for \ any \ irreducible \ component} \ W \ {\it of} \ V \} . \]
Clearly, $J$ is a finite index normal subgroup of $G$ and so is $G_0$ (Remarks \ref{rem:conn_char} and \ref{rem:g0}). 
It follows that $H$ is a finite index normal subgroup of $G$.
Therefore, $\overline{H}$ contains $\overline{G}_0$ by Remark \ref{rem:fin_ind}.
Since $H < \overline{G}_0$ and the latter group is pro-algebraic (Remark \ref{rem:g0}), 
we deduce $\overline{H} < \overline{G}_0$ and hence  $\overline{H} = \overline{G}_0$.
Remark \ref{rem:closure_structure} implies that $\overline{G}_0$ is solvable. 
As a consequence, $\overline{H}$ and its subgroup $H$ are also solvable.

Fix an irreducible component $W$ of $V$. Denote by $\ell$ the derived length of $H_{|W}$; 
it does not depend on $W$ since $H \lhd G$ and the $G$-orbit of $W$ is $V$. 
If $\ell = 0$ then we are done since $H_{|V} =\{  \mathrm{id}_{V} \}$ and $G_{|V}$ is finite. 
So we can assume $\ell \geq 1$ for any choice of $V$.

Denote $\mathcal{A}_{W} = \mathcal{A}_{H, W}$ and assume 
$\mathcal{A}_{W} \neq \emptyset$. 
In this case, we can replace $V$ with a subvariety $V'$ of smaller but still positive dimension. 
Indeed, Lemma \ref{lem:red_dim} provides 
a $H$-invariant analytic set $V''$ such that $1 \leq \dim (V'') < \dim (V)$.
By considering the $G$-orbit $V'$ of one of the irreducible components of $V''$, we obtain that
$V'$ is a $G$-invariant germ of analytic set and $1 \leq \dim (V') < \dim (V)$.
This conclusion is always achieved if $\ell \geq 2$ by Lemma \ref{lem:exist_infinite}.

Now consider the $G$-invariant analytic set $(\cc^{n}, {\bf 0})$. 
We can apply the previous method and since
the dimension can not decrease indefinitely, 
we obtain a germ of analytic set $V$, such that $V$ is the $G$-orbit of one of its irreducible components, 
$\ell =1$ and $\mathcal{A}_{W} = \emptyset$ for any irreducible component $W$ of $V$.
These properties imply that $H_{|V}$ is abelian and consists of finite order elements. 
Since $|G:H|< \infty$, $G_{|V}$  consists of finite order elements. 
\end{proof}

Now, we prove the main results for the torsion locus of a finite orbits subgroup of $\diff{}{n}$.

\begin{proof}[Proof of Theorem \ref{thm:main_g}]
Consider the germ of $G$-invariant germ of analytic variety $V$ provided by Proposition  \ref{pro:main}.
Since $V \subset \fg$ by Proposition \ref{pro:fg}, it follows that the torsion locus $\fg$ of $G$ has
positive dimension. The group $G$ is virtually solvable by Proposition  \ref{pro:main}. The remaining properties
are inherited from  Proposition \ref{pro:fg}.
\end{proof}

\begin{proof}[Proof of Theorem \ref{thm:main_g_fg}]
    It is a consequence of Theorem \ref{thm:main_g} and Proposition \ref{pro:fgfg}.
    More precisely, since $\fg = \fix (\overline{G}_{0} \cap G)$, the group $G_{|\fg}$ is a quotient of $G/ (\overline{G}_{0} \cap G)$.
    It is also a quotient of $\overline{G}/\overline{G}_0$ (Remark \ref{rem:conn_char}), a quotient of 
    $\overline{D_{\bf 0} G}/ \overline{D_{\bf 0} G}_{0}$ (Remark \ref{rem:g0}) and anti-isomorphic to a quotient of 
    $\overline{G}_{1}/\overline{G}_{1,0}$ (Remark \ref{rem:g0}).
\end{proof}
\begin{proof}[Proof of Corollary \ref{cor:main_g}]
The first case is a consequence of Theorem \ref{thm:main_g_fg}.
The second case derives from Theorem \ref{thm:main_g}, the definition of $\fg$ and that $\mm (\phi)=1$ for any $\phi \in G$ 
(Lemma \ref{lem:connected} and Definition \ref{def:embed}). The last case is a consequence of the second since 
$\spec (D_{\bf 0} \phi) =\{1\}$ implies that Condition \eqref{equ:cond_con} holds and hence $\mm (\phi)=1$.
\end{proof}

Proposition \ref{pro:fg} is a result of analytic maximality for the torsion locus. It is also maximal from a topological point of 
view in the finite orbits case:

\begin{pro}
    Let $G$ be a finite orbits subgroup of $\diff{}{n}$ with a finite symmetric generating set $\mathcal{S} = {\{ \phi_j \}}_{1 \leq j \leq k}$.
    Fix a representative $\varphi_j : U_j \to V_j$  of $\phi_j$, for any $1 \leq j \leq k$, 
    such that the pseudogroup $\mathcal{G}$ generated by ${\{ \varphi_j \}}_{1 \leq j \leq k}$ has finite orbits. 
    Suppose that there exists a totally $\mathcal{G}$-invariant continuum $\mathcal{K}$, i.e.   
    $\varphi_j (\mathcal{K}) = \mathcal{K}$ for any $1 \leq j \leq k$, with ${\bf 0} \in \mathcal{K}$.
    Then the germ $(\mathcal{K}, {\bf 0})$ is contained in $\fg$.
\end{pro}
Note that the total invariance of $\mathcal{K}$ implies 
$\mathcal{K} \subset \cap_{j=1}^{k} (U_j \cap V_j)$ in particular. 
\begin{proof}
Fix $\rho \in G$. There exists a word $\phi_{i_l} \circ \hdots \circ \phi_{i_1}$ such that $\rho = \phi_{i_l} \circ \hdots \circ \phi_1$.
The word $\varphi_{i_l} \circ \hdots \circ \varphi_1$ defines an element $\varrho: U \to V$ of $\mathcal{G}$, 
where $\mathcal{K} \subset U \cap V$ and $\varrho (\mathcal{K}) = \mathcal{K}$. 
Moreover, 
we can choose a small neighborhood $W$ of $\mathcal{K}$, with real analytic boundary, 
such that $\overline{W} \subset U \cap \rho (U)$. 
Analogously as in Proposition \ref{pro:cc_is_periodic}, replacing $\overline{B}(0,\epsilon)$ with $\overline{W}$, 
we obtain that  $\mathcal{K} \subset \fix (\varrho^{\mm (\rho)})$.
Thus, the germ $(\mathcal{K}, {\bf 0})$ is contained in the germ $\fix (\rho^{\mm (\rho)})$ at ${\bf 0}$.

There exist $\rho_1, \hdots, \rho_r$ such that $\fg = \cap_{j=1}^{r} \fix (\rho_{j}^{\mm (\rho_{j})})$ by noetherianity of $\Oc_{n}$.
Therefore, $(\mathcal{K}, {\bf 0})$ is contained in the torsion locus.  
\end{proof}

\section{Closed leaves foliations}
\label{sec:closed_leaves}
We already proved the main results for groups of local biholomorphisms (Theorems \ref{thm:main_g} and \ref{thm:main_g_fg}).
Now, we apply them to foliations  to obtain Theorem  \ref{thm:main}. 

\subsection{Closed leaves foliations and finite orbits}
As a first step in our goal of proving Theorem  \ref{thm:main}, we relate the closed leaves property for foliations with 
the finite orbits properties for groups. 
Let $\mathcal{F}$ be a holomorphic foliation in a complex manifold $M$. 
\begin{defi}
    \label{def:saturated}
    Given a subset $S$ of $M$, the {\it saturated} of $S$ is the minimal $\mathcal{F}$-invariant set containing $S$. 
\end{defi} 
\begin{defi}
    Given a subset $U$ of $M$, we denote by $\mathcal{F}_U$ the restriction of $\mathcal{F}$ to $U$.
    We consider either open or invariant subsets of $M$. 
\end{defi} 
\begin{defi}
    \label{def:stable}
    Let $\mathcal{L}$ be a leaf of $\mathcal{F}$. We say that $\mathcal{L}$ is {\it stable} if there exists a 
    neighborhood basis of $\mathcal{L}$ whose sets are $\mathcal{F}$-invariant. 
\end{defi}

\begin{lemma}
    Suppose that all leaves of $\mathcal{F}_U$ are closed in an open set $U$. Let $U'$ be an open subset of $U$. Then 
    all leaves of $\mathcal{F}_{U'}$ are closed in $U'$.  
\end{lemma}
\begin{proof}
By transverse uniformity of $\mathcal{F}$, a leaf $\mathcal{L}$ of $\mathcal{F}_U$ is closed if and only if $\mathcal{L}$ 
is a closed analytic $\mathcal{F}_{U}$-invariant subset of $U$ of 
dimension $\dim (\mathcal{F})$ \cite[section 3.3, Theorem 5]{Camacho-Lins_Neto:foliations}. Let ${\bf z} \in U'$ and denote by $\mathcal{L}_{\bf z}$ and $\mathcal{L}_{\bf z}'$
the leaves through ${\bf z}$ of $\mathcal{F}_{U}$ and $\mathcal{F}_{U'}$ respectively. By hypothesis  $\mathcal{L}_{\bf z}$ is closed in $U$. 
Moreover $\mathcal{L}_{\bf z}'$ is closed in $U'$ since $\mathcal{L}_{\bf z}'$ is a connected component of $\mathcal{L}_{\bf z} \cap U'$
and, in particular, a closed analytic subset of $U'$ of 
dimension $\dim (\mathcal{F})$.
\end{proof}
\begin{defi}
    \label{def:hol_pseudo}
    Consider an open subset $U$ of $M$ and a connected transverse section $T$ to $\mathcal{F}$.
    Let us consider the set $\mathfrak{G}$ of germs of maps $\phi: (T,{\bf z}) \to (T, \phi({\bf z}))$, 
    where  there exists a path $\gamma:[0,1] \to M$ 
    in the leaf of $\mathcal{F}_{U}$ through ${\bf z}$ such that $\gamma (0) = {\bf z}$, $\gamma (1) = \phi ({\bf z})$ and 
    $\phi$ is the holonomy map associated to $\gamma$. We define the {\it holonomy pseudogroup} $\mathcal{H}_{\mathcal{F},U,T}$ 
    of $\mathcal{F}_{U}$ in $T$ as the pseudogroup 
    generated by the maps of the form $\phi : W_1 \to W_2$,  where $W_1$ and $W_2$ are open subsets of $T$,
    and the germ $\phi_{\bf z}$ belongs to $\mathfrak{G}$ for any ${\bf z} \in W_1$. 
\end{defi}
\begin{defi}
    Consider the setting in Definition \ref{def:hol_pseudo}. Let $\mathcal{L}$ be a leaf  of $\mathcal{F}_{U}$ and 
    ${\bf z}_{0} \in \mathcal{L} \cap T$.  Let $\mathcal{H}_{\mathcal{F},\mathcal{L},T,{\bf z}_{0}}$ 
    be the holonomy subgroup ${(\mathcal{H}_{\mathcal{F},U,T})}_{{\bf z}_{0}}$
    of  $\mathrm{Diff} (T, {\bf z}_{0})$ (cf. Definition \ref{def:germ}).
    We denote by $\mathcal{H}_{\mathcal{F},\mathcal{L}}$ if  $T$ and ${\bf z}_{0}$ are implicit.
\end{defi}  
\begin{rem}
    $\mathcal{H}_{\mathcal{F},\mathcal{L}}$ is the image of the holonomy morphism 
    $\pi_{1} ({\mathcal L},{\bf z}_{0}) \to \mathrm{Diff} (T, {\bf z}_{0})$.
\end{rem}

\begin{pro}
\label{pro:closed_finite}
Let $\mathcal{F}$ be a holomorphic foliation in a complex manifold $M$ and $\mathcal{L}$ be a compact leaf of $\mathcal{F}$. 
Consider a connected holomorphic transverse section $T$ to $\mathcal{F}$ containing a point ${\bf z}_{0} \in \mathcal{L}$. 
Then there exists an open neighborhood $U$ of $\mathcal{L}$ such that $\mathcal{F}_{U}$ has closed leaves if and only if 
$\mathcal{H}_{\mathcal{F},\mathcal{L},T,{\bf z}_{0}}$ has the finite orbits property.
\end{pro}
\begin{proof}
Consider a small open neighborhood $U$ of $\mathcal{L}$ such that $\mathcal{F}_{U}$ has closed leaves. 
Any leaf $\mathcal{L}'$ of $\mathcal{F}_U$ intersects $T \cap U$ in a discrete closed set  \cite[section 3.3, Theorem 5]{Camacho-Lins_Neto:foliations}. 
Consider an open neighborhood $U'$ of
$\mathcal{L}$ such that $U'$ is relatively compact and 
$\overline{U'} \subset U$. Then $\mathcal{H}_{\mathcal{F},U',T}$ has finite orbits and hence 
$\mathcal{H}_{\mathcal{F},\mathcal{L}}$ has the finite orbits property.

Suppose that $\mathcal{H}_{\mathcal{F},\mathcal{L}}$ has finite orbits.
Consider generators $[\gamma_1], \hdots, [\gamma_{\ell}]$ of $\pi_{1} (\mathcal L, {\bf z}_{0})$. 
Denote by $\phi_{j}: U_j \to V_j$, where ${\bf z}_{0} \in U_j \cap V_j$ and $U_j \cup V_j \subset T$, 
a representative of the holonomy map associated to $[\gamma_j]$
for $1 \leq j \leq \ell$. Let $\mathcal{H}$ be the pseudogroup generated by ${\{ \phi_j : U_j \to V_j \}}_{1 \leq j \leq \ell}$.
By hypothesis, and up to a good choice of representatives, $\mathcal{H}$ has finite orbits.

Fix a Riemannian metric in $\mathcal{L}$ that we lift to its universal cover $\tilde{\mathcal L}$.
Then there exists $C \in \mathbb{R}^{+}$ such that 
for any closed path $\gamma: [0,1] \to \mathcal{L}$, with $\gamma (0) = \gamma (1) = {\bf z}_{0}$, 
and up to a reparametrization of $\gamma$, there exists a word $\mathcal{W}$ on 
$\gamma_1, \hdots, \gamma_{\ell}$ such that $\gamma$ and $\mathcal{W}$ are homotopic rel $0,1$
via a (uniformly) bounded diameter homotopy $F:[0,1]\times[0,1] \to \mathcal{L}$, i.e.  
\begin{itemize}
\item $t \mapsto F(t,0)$ parametrizes $\gamma$ and $t \mapsto F(t,1)$ parametrizes $\mathcal{W}$;
\item any lift of $s \mapsto F(t,s)$ to  $\tilde{\mathcal L}$   is contained in 
a set of diameter less than $C$ for any $t \in [0,1]$.
\end{itemize}
This property guarantees that  for any small open 
neighborhood $U'$ of $\mathcal{L}$,  any orbit of $\mathcal{H}_{\mathcal{F},U',T}$ is
contained in the orbit of $\mathcal{H}$.

Given a small open neighborhood $U'$ of $\mathcal{L}$, consider a smaller open neighborhood $U$ of $\mathcal{L}$ contained in $U'$. 
Suppose, aiming at contradiction, that there exists a non-closed leaf $\mathcal{L}_{{\bf z}_{1}}$ of $\mathcal{F}_{U}$ through ${\bf z}_{1}$.
Then there exists an infinite orbit of $\mathcal{H}_{\mathcal{F},U',T}$.
Since it is contained in an orbit of $\mathcal{H}$, this contradicts that $\mathcal{H}$ has finite orbits.
\end{proof}

Now, we can show the main result of the paper. 

\begin{proof}[Proof of Theorem \ref{thm:main}]
Fix ${\bf z}_{0} \in \mathcal{L}$ and a transverse section $T$ to $\mathcal{F}$ such that ${\bf z}_{0} \in T$.
Since there exists an open neighborhood $U$ of $\mathcal{L}$ where $\mathcal{F}_{U}$ has closed leaves, 
the holonomy group $\mathcal{H}_{\mathcal{F},\mathcal{L}}$ of $\mathcal{L}$ has finite orbits 
by Proposition \ref{pro:closed_finite}.  We apply Theorem \ref{thm:main_g_fg} to the finitely generated group 
$\mathcal{H}_{\mathcal{F},\mathcal{L}}$ to obtain a torsion locus $\mathfrak{T} := \mathfrak{T}_{\mathcal{H}_{\mathcal{F},\mathcal{L}}}$
in $T$ at ${\bf z}_{0}$
of positive dimension such that the action of $\mathcal{H}_{\mathcal{F},\mathcal{L}}$ in $\mathfrak{T}$ is finite.
Reeb local stability theorem implies that the saturated $V$ of $\mathfrak{T}$ with respect to $\mathcal{F}$ 
(cf. Definition \ref{def:saturated})
is a germ of $\mathcal{F}$-invariant set in a neighborhood of 
$\mathcal{L}$ of dimension $\dim ({\mathcal F}) + \dim (\mathfrak{T})$, consisting of compact leaves and such that  
the holonomy group of $\mathcal{F}_{V}$ of every leaf of $V$ is finite.
In particular all leaves of $\mathcal{F}_{V}$ are stable.
\end{proof}

\subsection{Closed leaves foliations in projective manifolds}
This section is devoted to show Theorem  \ref{thm:main_p}.

Let $V$ the $\mathcal{F}$-invariant variety provided by Theorem \ref{thm:main}. 
It is the saturated of the torsion locus $\mathfrak{T} := \mathfrak{T}_{\mathcal{H}_{\mathcal{F},\mathcal{L}}}$
with respect to $\mathcal{F}$.
Our goal is globalizing $V$. Consider a holomorphic transverse section such that $\mathcal{L} \cap T = \{ {\bf z}_{0} \}$.

By work of G\'{o}mez-Mont \cite[Theorem 3]{Gomez-Mont:integral_compact}, there exists a countable family 
${\{S_l\}}_{l \in \mathcal{S}}$ of  irreducible Zariski-closed subvarieties of the Hilbert scheme  of varieties of dimension
equal to $\codim (\mathcal{F})$ such that
the closure of any quasi-projective leaf of $M$ belongs to $\cup_{l \in \mathcal{S}} S_l$ and 
each $S_l \in \mathcal{S}$ satisfies
\begin{itemize}
    \item there exists a universal projective subscheme $\mathcal{U}_l \subset M \times S_l$ such that 
    $p_{2,l}^{-1} (s)$ is the scheme represented by $s$ for $s \in S_l$
    and $p_{2,l}$ is a flat proper morphism,
    where $p_{1,l}: \mathcal{U}_l \to M$ and $p_{2,l}: \mathcal{U}_l \to S_l$
    are the natural (proper) projections;
    \item $p_{1,l} :  \mathcal{U}_l  \to W^{l}$ is a birational proper morphism, where  $W^{l} =  p_{1,l} (\mathcal{U}_l)$;
     \item $f_l: W^{l} \to S_l$ is a dominant rational map, where $f_{l}: W^{l} \to S_l$ is defined by $f_l = p_{2,l} \circ p_{1,l}^{-1}$;
    \item $p_{2,l}^{-1} (s)$ is $\mathcal{F}$-invariant for any $s \in S_l$  
    and in particular $W^{l} $ is $\mathcal{F}$-invariant;
    \item the set $T_l$ of $s \in S_l$ such that $p_{2, l}^{-1} (s)$ is integral is a non-empty Zariski-open subset of $S_l$
    and the set $T_{l}'$ of $s \in T_l$ such that $p_{2, l}^{-1} (s) \cap \mathrm{Sing} (\mathcal{F}) = \emptyset$
    is a Zariski-open subset of $S_l$ since it is $S_{l} \setminus p_{2,l}(p_{1,l}^{-1} (\mathrm{Sing} (\mathcal{F})))$.
\end{itemize}
Note that $f_l$ is a first integral by the fourth bullet point. 

Let ${V}_j$ be an irreducible component of the germ $(V, \mathcal{L})$. 
We claim that there is some $\ell \in \mathcal{S}$ such that $V_j \subset W^{\ell}$;
otherwise we get $\dim (V_j \cap W^{\ell}) < \dim (V_j)$ for any $l \in \mathcal{S}$, 
contradicting $V_j \subset \cup_{l \in \mathcal{S}} (V_j \cap W^{l})$ since $\mathcal{S}$ is countable. 
In particular $S_{\ell}$ has positive dimension.
Since $\mathcal{L} \in S_{\ell}$ and $\mathcal{L} \cap \mathrm{Sing} (\mathcal{F}) = \emptyset$,
the general $f_{\ell}$-fiber is a leaf of $\mathcal{F}$ by the last three bullet points.

The variety $W_j := W^{\ell}$ is projective since $p_{1,l}: \mathcal{U}_l \to M$ is proper. 
By using the closedness of the Fubini-Study volume and Stokes theorem, we deduce that the volume $s \mapsto \mathrm{vol} (p_{2,\ell}^{-1} (s))$
is constant in $T_{\ell}'$. As a consequence, the function $s \mapsto \mathrm{vol} (p_{2, \ell}^{-1} (s))$ is bounded in $S_{\ell}$ and hence
all holonomy groups of leaves of  $\mathcal{F}_{W_j}$ are finite. 
Therefore, we deduce $(W_j \cap T, {\bf z}_{0}) \subset (\mathfrak{T}, {\bf z}_{0})$ 
by Theorem \ref{thm:main_g_fg}. This implies $({W}_{j}, \mathcal{L}) \subset (V , \mathcal{L})$ by Theorem \ref{thm:main}.
 
Consider the irreducible components ${V}_1, \hdots, {V}_k$ of the germ $(V, \mathcal{L})$. 
Since 
\[  (V, \mathcal{L}) = ({V}_1 \cup \hdots \cup {V}_{k}, \mathcal{L}) \subset ({W}_1 \cup \hdots \cup {W}_{k}, \mathcal{L}) \subset (V, \mathcal{L}),  \]
we deduce 
$({W}_1 \cup \hdots \cup {W}_{k}, \mathcal{L}) = (V, \mathcal{L})$. 
Since $(V, \mathcal{L})$ is the saturated of $\mathfrak{T}$ by Theorem \ref{thm:main},
this concludes the proof of Theorem \ref{thm:main_p}.

\section{Finite orbits pseudo-group inducing an infinite subgroup of $\diff{}{n}$}  
\label{section:infinite}
Let us present examples of finite orbits subgroups of $\diff{}{}$ that satisfy the thesis of Theorem \ref{thm:main_g}
but not the conclusion of Theorem \ref{thm:main_g_fg}. Of course, such examples are not finitely generated. 

The first example is very simple. We define $\lambda_k = e^{\frac{2 \pi i}{2^{k-1}}}$ and 
\[ 
\begin{array}{ccccc}
\sigma_k & : & B(0,1/k) & \to  & B(0,1/k) \\
           &  & z & \mapsto & \lambda_k z
\end{array}
\]
for $k \in \nn$.
%
%
\begin{rem}
    \label{rem:inf_simple}
    The pseudogroup $\mathcal{P} := \langle \sigma_1, \hdots, \sigma_k, \hdots \rangle$ generated by the family ${(\sigma_{k})}_{k \geq 1}$ has finite orbits 
    and thus $P := \mathcal{P}_{\bf 0}$ is an infinite subgroup of $\diff{}{}$ that has the finite orbits property. The germ $\mathfrak{T}_{P}$ 
    is equal to $(\mathbb{C}, {\bf 0})$.
    In particular $P$ is an infinite abelian group consisting of finite order elements. 
    So the finite generation hypothesis is necessary in Theorem \ref{thm:main_g_fg}.
\end{rem} 
Unfortunately the above subgroup $P$ is not a ``natural" example since it was built by restricting the natural domains of definition of the 
diffeomorphisms in the family ${(\sigma_{k})}_{k \geq 1}$. 
Indeed, it would not be an example if we would have required pseudogroups (cf. Definition \ref{def:pseudogroup}) 
to be closed under analytic continuation.
This section is devoted to obtain a holomorphic pseudogroup with analogous properties as $P$ and also closed under analytic continuation.
The example is provided by next propositions.  
\begin{pro}
    \label{pro:natural_aux}
    There exists a sequence ${\{ \phi_k : \overline{D_k} \to \overline{D_k}\}}_{k \geq 1}$ of homeomorphisms such that
    \begin{enumerate}     
        \item \label{cond2}  $D_k$ (resp. $\overline{D_k}$) is an open (resp. closed) topological disc with $0 \in D_k \subset \cc$;
        \item \label{cond3}  
        $(\phi_k)_{|D_k} : D_k \to D_k$ is a biholomorphism such that $\phi_{k}' (0) = \lambda_k$;
        \item \label{cond4} $\phi_k$ cannot be holomorphically continued past any point of $\partial D_k$;
        \item \label{cond5} $\overline{D_{k}} \subset D_{k-1}$ and $\phi_{k}^{2} = {(\phi_{k-1})}_{|\overline{D}_{k}}$   
    \end{enumerate}
    for any $k  \geq 2$.
    Moreover, we require  $D_{1}= \cc$,  $\phi_{1} = \mathrm{id}$ and $\lim_{k \to \infty} \mathrm{diam} (\overline{D_{k}})  = 0$.
\end{pro}
\begin{pro}
    Let $\mathcal{G}$ be the holomorphic pseudogroup generated by ${\{ \phi_k: D_k \to D_k \}}_{k \in \nn}$.
    Then $\mathcal{G}$ is closed under analytic continuation,  has finite orbits and 
    $G:=\mathcal{G}_0$ is an infinite group that has the finite orbits property.
\end{pro}
\begin{proof}
    Let us assume that Proposition \ref{pro:natural_aux} holds.
    The map $\phi_{k}$ is conjugated to the rotation $z \mapsto \lambda_k z$ in the closed disc $\overline{D_k}$ and 
    $D_k$ is the natural domain of definition of the holomorphic map $\phi_k$  by Condition \eqref{cond4}.
    As a consequence of Condition \eqref{cond5}, the 
    orbits of $\mathcal{G}$ in $D_{k} \setminus D_{k+1}$ coincide with the orbits of $\langle \phi_k \rangle$ 
    and thus all them have $2^{k-1}$ elements for any $k \in \nn$. 
    Since $\lim_{k \to \infty} \mathrm{diam} (\overline{D_{k}})  = 0$, 
    $\mathcal{G}$ has finite orbits and $G$ has the finite orbits property.
\end{proof}
Let us introduce the setting for the proof of Proposition \ref{pro:natural_aux}.
Consider the map $\pi_k (z) =  z^{2^{k-1}}$ defined in $\cc$ for $k \in \nn$. 
The following result is key to obtain Condition \ref{cond4}.
\begin{lemma}
    \label{lem:2arcs}
    There exist a homeomorphism $\tau: \overline{D} \to \overline{B}(0, c)$ such that 
    $D$ and $\overline{D}$ are an open and a closed topological disc respectively containing $0$, 
    and closed arcs $\alpha$ and $\beta$ in $\partial D$ such that 
\begin{itemize}
    \item $\tau_{|D}$ is holomorphic and $\tau' (0)=1$;
    \item $\alpha \cup \beta = \partial D$ and $\sharp (\alpha \cap \beta) =2$;
    \item $\tau (\alpha)$ and $\tau (\beta)$ are arcs in $\partial B(0,c)$ of equal length $\pi c$; 
    \item the interior of $\alpha$ is a real analytic curve;
    \item $\beta$ is not real analytic at any point.
\end{itemize}
\end{lemma}
\begin{proof}
    Let $\beta'$ be a simple smooth path that is nowhere real analytic \cite{Fabius:nowhere_analytic}. 
    Now consider a simple path $\alpha'$ such that the interior of $\alpha'$ is real analytic, 
    $\alpha' \cup \beta'$ is  a simple closed path and $\sharp (\alpha' \cap \beta') =2$.
    Let $D'$ be the bounded connected component of $\C \setminus (\alpha' \cup \beta')$; it is an open disc
    whose closure $\overline{D'}$ is a closed topological disc such that $\partial{D'} = \alpha' \cup \beta'$.

    Consider the Riemann map  $\theta: D' \to {B}(0,1)$; it extends to a homeomorphism $\overline{D'} \to \overline{B}(0,1)$
    since $\partial D'$ is a simple closed path.
    The sets $\theta (\alpha')$ and $\theta (\beta')$ are arcs. Thus, there exists a M\"{o}bius transformation $\sigma$ such that 
    $\sigma (B(0,1)) = B(0,1)$ and $\sigma (\theta (\alpha'))$ and $\sigma (\theta (\beta'))$ are arcs of length $\pi$.
    Now, consider $z_0 = (\sigma \circ \theta)^{-1}(0)$ and define
    \[ \tau = (\mu z) \circ \sigma \circ \theta \circ (z+z_0) \]
    where $\mu = 1/ (\sigma \circ \theta )' (z_0)$. By construction, we obtain $\tau'(0)=1$. We denote 
    $D = D' - z_{0}$, $\alpha = \alpha' - z_0$,   $\beta = \beta' - z_0$ and $c = |\mu|$.
    Clearly $\tau: \overline{D} \to \overline{B}(0, c)$ satisfies all conditions.
\end{proof}
   Since we want $\lim_{k \to \infty} \mathrm{diam} (\overline{D_{k}})  = 0$ in Proposition \ref{pro:natural_aux}, we are
   going to consider scaled versions of $\tau$ defined in small domains. 
\begin{defi}
    Consider a sequence ${(\epsilon_k)}_{k \geq 1}$ of positive real numbers to be determined. We define
    $\tau_k: \epsilon_k \overline{D} \to B(0, \epsilon_k c)$ by the formula $\tau_k (z) = \epsilon_k \tau (z/\epsilon_k)$
    for $k \in \nn$.
\end{defi}
\begin{lemma}
    \label{lem:tau_noext}
    $\tau^{-1} \circ (-z) \circ \tau$ cannot be extended holomorphically past any point in $\partial D$.
    In particular, $\tau_{k}^{-1} \circ (-z) \circ \tau_{k}$ cannot be extended holomorphically past any point in $\partial (\epsilon_k D)$
    for any $k \in \nn$.
\end{lemma}
\begin{proof}
    Denote $\sigma = \tau^{-1} \circ (-z) \circ \tau$ and consider the notations in Lemma \ref{lem:2arcs}.
    Note that since $\tau (\alpha)$ and $\tau (\beta)$ have the same length, we obtain $\sigma (\alpha) = \beta$ and 
    $\sigma (\beta) = \alpha$. Since the interior of $\alpha$ is real analytic and $\beta$ is not real analytic at any point, 
    $\sigma$ can not be extended holomorphically past any point of $\partial D$.
\end{proof}
     Now, we lift the construction in Lemma \ref{lem:tau_noext} by $\pi_k$.
\begin{defi}
Given $k \in \nn$ we define $E_k = \pi_{k}^{-1} (\epsilon_k D)$, the unique $\delta_k \in \mathbb{R}^{+}$ with 
$\pi_{k} (\delta_k) = c \epsilon_k$ and the unique map $\rho_k : \overline{E_k} \to  \overline{B}(0, \delta_k)$ such that following diagram commutes
\[
\begin{tikzcd}
 \overline{E_k} \arrow{r}{\rho_{k}} \arrow[swap]{d}{\pi_k} & \overline{B}(0, \delta_k) \arrow{d}{\pi_k} \\
\overline{\epsilon_k D} \arrow{r}{\tau_k} &\overline{B}(0, c \epsilon_k)
\end{tikzcd}
\]
and $\rho_{k}'(0)=1$. 
\end{defi} 
\begin{lemma}
     \label{lem:rho_aux}
    Consider a sequence ${(\epsilon_k)}_{k \geq 1}$ in $\mathbb{R}^{+}$ such that $\epsilon_{k+1} << \epsilon_k$ for any $k \in \nn$. Then 
    \begin{enumerate}[(a)]
        \item \label{conda}$E_k$ (resp. $\overline{E_k}$) is an open (resp. closed)  topological disc such that $0 \in E_k$;
        \item \label{condb} $(\rho_k)_{|E_k} : E_k \to B (0, \delta_k)$ is a biholomorphism; 
        \item \label{condc} $\rho_k$ cannot be holomorphically continued past any point of $\partial E_k$;
        \item \label{conde} $\rho_{k}^{-1} (\overline{E_{k+1}}) \subset E_{k}$;
        \item \label{condd} $\rho_k$ is of the form $z h_{k} (\pi_k(z))$ where $h_k$ is continuous in $\pi_k (\overline{E_k})$, holomorphic
    in $\pi_{k} (E_k)$ and $h_{k}(0) = 1$
\end{enumerate}
    for any $k \in \nn$. 
\end{lemma}
\begin{proof}
    Conditions \ref{conda} and \ref{condd} hold by construction.
    The map $\rho_k :  \overline{E_k}  \to \overline{B}(0, \delta_k)$ is a local homeomorphism that is holomorphic in $E_k$ for any $k \in \nn$. 
    Since $\rho_k$ is a covering and $\overline{B}(0, \delta_k)$ is simply connected, it follows that $\rho_k$ is a homeomorphism.
    So, Condition \ref{condb} is satisfied.
    Condition \ref{condc} derives immediately from Lemma \ref{lem:tau_noext}. We obtain Condition \ref{conde} by requiring 
    $\epsilon_{k+1} << \epsilon_k$.
\end{proof}
Up to conjugacy, the family ${\{ \phi_k \}}_{k \geq 1}$ in Proposition \ref{pro:natural_aux} is given by next result. 
\begin{lemma}
\label{lem:nopast}
$\rho_{k}^{-1} \circ (\lambda_{k+1} z) \circ \rho_k: E_k \to E_k$ does not extend holomorphically past any point in 
$\partial{E_k}$ for any $k \in \nn$.
\end{lemma}
\begin{proof}
Since the diagram 
\[
\begin{tikzcd}
 \overline{E_k} \arrow{r}{\rho_{k}} \arrow[swap]{d}{\pi_k} & \overline{B}(0, \delta_k)  \arrow{r}{\lambda_{k+1} z} \arrow{d}{\pi_k} &   
 \overline{B}(0, \delta_k) \arrow{r}{\rho_{k}^{-1}}  \arrow{d}{\pi_k} &  \overline{E_k} \arrow{d}{\pi_k}  \\
\overline{\epsilon_k D} \arrow{r}{\tau_k} &\overline{B}(0, c \epsilon_k)  \arrow{r}{-z} & \overline{B}(0, c \epsilon_k)  \arrow{r}{\tau_k^{-1}} & \overline{\epsilon_k D}
\end{tikzcd}
\]
commutes, $\rho_{k}^{-1} \circ (\lambda_{k+1} z) \circ \rho_k$ does not extend holomorphically past any point in 
$\partial{E_k}$ by Lemma \ref{lem:tau_noext}.
\end{proof}
\begin{defi}
We define $\psi_k = \rho_k \circ \hdots \circ \rho_1$, it is defined in 
$D_{k+1} := \psi_{k-1}^{-1}(E_k)$ for any $k \in \nn$ where $\psi_{0} =z$. We define
\[ \phi_{k} = \psi_{k-1}^{-1} \circ (\lambda_{k} z) \circ \psi_{k-1} \]
for $k \in \nn$.
\end{defi}
\begin{rem}
Since $\rho_k (z) = z (1 +O(z^{2^{k-1}}))$, it follows that ${(\psi_k)}_{k \geq 1}$ converges in the Krull topology (cf. Definition \ref{def:krull}) to 
some $\psi \in \diffh{}{}$.  
\end{rem}
The group $G$ can be formally linearized (see \cite{Lamy:parabolic} for a similar construction in a different context).
\begin{lemma}
We have $\phi_k =  \psi^{-1} \circ (\lambda_k z) \circ \psi$ in $\diffh{}{}$ for any $k \in \nn$. In particular, 
we obtain $G = \psi^{-1} \circ P \circ \psi \subset \diff{}{}$, 
where $P = \langle \lambda_1 z, \hdots, \lambda_k z, \hdots \rangle$. 
\end{lemma}
\begin{proof}
Since
\[ [\rho_j \circ (\lambda_k z)](z) =
\rho_j (\lambda_k z) = \lambda_k z h_{j} (\pi_j (\lambda_k  z)) =  \lambda_k  z h_{j} (\pi_j (z)) = [(\lambda_k z) \circ \rho_j](z) \]
for $j \geq k$, it follows that 
\[  \phi_{k} = \psi_{k-1}^{-1} \circ (\lambda_{k} z) \circ \psi_{k-1} =  \psi_{j}^{-1} \circ (\lambda_{k} z) \circ \psi_{j} \]
for any $j \geq k-1$ and thus $\phi_{k} =  \psi^{-1} \circ \sigma_{k} \circ \psi$ for any $k \in \nn$.
\end{proof}
\begin{proof}[Proof of Proposition \ref{pro:natural_aux}]
By considering a sequence  ${(\epsilon_k)}_{k \geq 1}$ in $\mathbb{R}^{+}$ converging sufficiently quickly to $0$, 
we obtain Conditions \ref{conda}-\ref{conde} (Lemma \ref{lem:rho_aux}).
Moreover, we have
\[ \phi_{k+1}^{2} = (\psi_{k}^{-1} \circ (\lambda_{k+1} z) \circ \psi_{k})^{2} =  \psi_{k}^{-1} \circ (\lambda_{k} z) \circ \psi_{k} = \psi_{k-1}^{-1} \circ (\lambda_{k} z) \circ \psi_{k-1} = \phi_k  \]
in $D_{k+1}$, since $\rho_{k}^{-1} \circ (\lambda_k z) \circ \rho_k = \lambda_k z$, for any $k \in \nn$.
Therefore, we obtain all properties in Proposition \ref{pro:natural_aux}.
\end{proof}

 {\it Declaration of Generative AI and AI-assisted technologies in the writing process}.
 During the preparation of this work the author used Gemini 3.8 Flash and 3.1 Pro in order to consult bibliography references and results already in the literature. 
 The results, proofs and writing of the paper are due to the author, who takes full responsibility for the content of the publication.

\bibliographystyle{alpha}		
\bibliography{rendu}

\end{document}